\documentclass[11pt, a4paper, final]{article}
\usepackage{amsfonts}
\usepackage{makeidx}
\usepackage[a4paper]{geometry}
\usepackage{amsmath}
\usepackage{amssymb}
\usepackage{amsthm}
\usepackage{bm}
\usepackage{cite}
\usepackage{showkeys}
\usepackage[english]{babel}
\usepackage{dblaccnt}
\usepackage{accents}
\usepackage{color}
\usepackage{graphicx}
\usepackage{psfrag}
\usepackage{subfigure}

\newcommand*{\N}{\ensuremath{\mathbb{N}}}
\newcommand*{\Z}{\ensuremath{\mathbb{Z}}}
\newcommand*{\R}{\ensuremath{\mathbb{R}}}
\newcommand*{\C}{\ensuremath{\mathbb{C}}}
\newtheorem{defi}{Definition}[section]
\newtheorem{lemma}[defi]{Lemma}
\newtheorem{theorem}[defi]{Theorem}
\newtheorem{corollary}[defi]{Corollary}

\newtheorem{remark}[defi]{Remark}
\renewcommand{\d}[1]{\,\mathrm{d}#1 \,}

\renewcommand{\O}{\mathcal{O}}
\newcommand{\ol}[1]{\overline{#1}}

\DeclareMathOperator{\supp}{\mathrm{supp}}

\usepackage{graphicx}

\usepackage{amscd}
\usepackage[mathscr]{euscript}

\DeclareMathOperator{\curl}{curl}

\renewcommand{\div}{\mathrm{div}}
\definecolor{lightblue}{rgb}{0.5,0.5,1}

\renewcommand{\Re}{\mathrm{Re}\,}
\renewcommand{\Im}{\mathrm{Im}\,}

\newcommand{\epsr}{\varepsilon_{\mathrm{r}}}

\newcommand{\ov}{\overline}

\newcommand{\pa}{\partial}
\newcommand{\nal}{\nabla_{\alpha}}
\newcommand{\nalT}{\nabla_{\alpha,T}}
\newcommand{\naT}{\nabla_T}
\newcommand{\intO}{\int_{\Omega}}

\begin{document}

\sloppy

\title{On Uniqueness in Electromagnetic Scattering\\ from Biperiodic Structures}
\author{Armin Lechleiter\thanks{Center for Industrial Mathematics, University of Bremen, 28359 Bremen, Germany, \texttt{lechleiter@math.uni-bremen.de}}
\and Dinh-Liem Nguyen\thanks{DEFI, INRIA Saclay--Ile-de-France and Ecole Polytechnique, Palaiseau, France, \texttt{dnguyen@cmap.polytechnique.fr}}} 

\maketitle

 \begin{abstract}
Consider time-harmonic electromagnetic wave scattering from a 
biperiodic dielectric  structure mounted on a perfectly conducting 
plate in three dimensions. Given that uniqueness of solution holds, 
existence of solution follows from a well-known Fredholm framework
for the variational formulation of the problem in a suitable Sobolev space. 
In this paper, we derive a Rellich identity for a solution to this 
variational problem under suitable smoothness conditions on the material 
parameter. Under additional non-trapping assumptions on the material 
parameter, this identity allows us to establish uniqueness of solution 
for all positive wave numbers.
% In this paper, we derive a Rellich identity 
% for a solution to this variational problem under suitable non-trapping and 
% smoothness conditions on the material parameter. This identity allows us 
% to establish uniqueness of solution for all wave numbers.
 \end{abstract}

\section{Introduction}

%{\color{blue}
Scattering of electromagnetic waves from periodic structures is not only an 
interesting mathematical topic in its own right but also of great interest in 
applications, e.g., for the construction and optimization of optical filters, 
lenses, and beam-splitters in optics. 
An overview about this and further topics in applied mathematics related to wave  
propagation in periodic structures can be found in, e.g.,~\cite{Bao2001}. 
In this paper we consider scattering of time-harmonic  
electromagnetic waves from a dielectric biperiodic structure mounted on a 
perfectly conducting plate in three dimensions. By biperiodic, we mean that 
the structure is periodic in the, say, $x_1$- and $x_2$-direction, while it 
is bounded in the $x_3$ direction. In contrast to scattering from bounded 
structures, uniqueness of solution for this scattering problem does  in general not 
hold for all positive wave numbers. 
Instead, non-trivial solutions to the homogeneous problem might exist for 
a discrete set of exceptional wave numbers, and these solutions turn out to be 
exponentially localized surface waves. Our study in the present paper focuses on
conditions guaranteeing the well-posedness of the full three-dimensional 
electromagnetic scattering problem mentioned above. We establish non-trapping 
and smoothness conditions on the (non-absorbing) dielectric such that uniqueness 
of solution holds for \emph{all} positive wave numbers. This means that materials 
satisfying the latter conditions cannot guide surface waves. 

% is motivated by the important applications 
% of biperiodic structures in diffractive optics, e.g., for optical filters, lenses, and beam-splitters. 
% A variety of further topics in applied mathematics related to light 
% propagation in periodic structures can be found in, e.g.,~\cite{Bao2001}.

\begin{figure}%[hbt]
\centering
\psfrag{x3}{$x_3$}
\psfrag{x1}{$x_2$}
\psfrag{x2}{$x_1$}
\includegraphics[width=8cm]{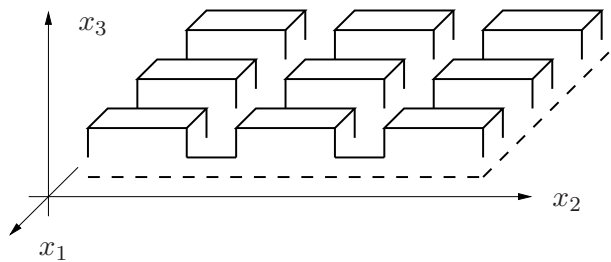}
\caption{Sketch of the biperiodic structure under consideration.}
\end{figure}
Mathematical formulation for the well-posedness of electromagnetic scattering problem 
for periodic structures has been an active area of research in the last years. 
For the scalar case, the authors in~\cite{Alber1979, Wilco1984} studied uniqueness
of solution for all wave numbers (or, equivalently, all frequencies), under geometrical 
conditions on the scatterer, 
for impenetrable structures with Dirichlet and Neumann conditions. Similar results
are obtained in the paper~\cite{Bonne1994} for more complicated periodic structures 
which are constituted of conducting and dielectric materials. The latter paper further 
gave examples of structures for which non-uniqueness of solution occurs at the 
so-called singular wave numbers. These wave numbers were shown to be related  
to guided waves (surface waves) that are exponentially localized along the structure.
%}

For the case of Maxwell's equations, the authors of~\cite{Dobso1992a} studied 
well-posedness of the scattering problem from a medium consisting of two 
homogeneous materials separated by a smooth biperiodic surface using an integral 
equation approach. In~\cite{Dobso1994, Bao1997, Bao2000} 
the authors studied existence and uniqueness of solution for the scattering problem
from penetrable biperiodic structures using a variational approach for the magnetic 
field. Nevertheless, unlike the scalar cases, the uniqueness results in cited cases of 
Maxwell's equations were proven for all but possibly a discrete set of wave numbers.
Furthermore, all the cited papers above considered the non-magnetic case, i.e,
the coefficient magnetic permeability is assumed to be the same constant outside and
inside the structure. The case of variable magnetic permeability were investigated in the 
paper~\cite{Abbou1993} for Maxwell's equations where the biperiodic structure consists of 
conducting and dielectric materials. That paper studied a variational approach, formulated in terms 
of the electric field, and showed that the obtained saddle point problem satisfies the Fredholm 
alternative, and again uniqueness of solution was proven for all but possibly a discrete 
set of wave numbers. More recently, the paper~\cite{Schmi2003} analyzed the well-posedness of 
the scattering problem for penetrable anisotropic biperiodic structures with a restriction on
the non-magnetic case again. The latter paper also proved that the scattering problem is uniquely 
solvable for all wave numbers if the structure contains absorbing materials, and if the 
dielectric tensor is piecewise analytic. Hence, to the best of our knowledge, 
uniqueness results for all wave numbers for the vectorial scattering problem
still remains open if the biperiodic materials is non-absorbing.

%{\color{blue}
The aim of the present work is to prove that the electromagnetic 
scattering problem for non-absorbing biperiodic dielectric structures mounted on a 
perfectly conducting plate is uniquely solvable for all positive wave numbers if
the material parameter satisfies non-trapping and smoothness conditions. 
% and perfectly conductors A method of variation of boundaries was 
% proposed in~\cite{Bruno1993} to study the numerical solution of the problem. 
% In~\cite{Dobso1992a}, the authors studied existence and uniqueness of the scattering problem 
% from a non-magnetic medium consisting of two homogeneous materials separated by a smooth
% biperiodic surface using an integral equation approach. This approach for the well-posedness 
% as well as the numerical solution for the scalar case has also been investigated 
% in~\cite{Arens2010}, see also~\cite{Raths2006}. 
% The paper~\cite{Abbou1993} considers a quite general biperiodic structure with a variable magnetic 
% permeability. This paper introduced a variational approach, formulated in terms 
% of the electric field, and showed that the obtained saddle point problem satisfies the Fredholm 
% alternative, and uniqueness of solution was proven for all but possibly a discrete set of wave numbers. 
% In~\cite{Dobso1994, Bao1997, Bao2000} the authors showed that a variational problem formulated 
% for the magnetic field is uniquely solvable for all but possibly a discrete set of wave numbers, 
% as well. Note that the above-mentioned results are also used to justify the stability of finite element 
% methods. 
We formulate the Maxwell's equations variationally in terms of the magnetic 
field in a suitable Sobolev space. We further restrict ourselves to the case 
of non-magnetic and isotropic materials. The variational problem is 
well-known to fit into a Fredholm framework, see, e.g.,~\cite{Dobso1994,Bao1997,Schmi2003}. 
(These papers deal with periodic scattering in the full space, but can be 
adapted to the half-space setting that we consider here.)
% However, as mentioned above, the uniqueness results in the cited papers 
% do not hold for all wave numbers if the material is non-absorbing.
As mentioned in the paper~\cite{Bonne1994} on the corresponding scalar scattering 
problems, non-uniqueness phenomena indeed arise at certain singular wave numbers  
if the non-absorbing material parameter 
satisfies suitable trapping conditions. In this paper we show a converse result for the full 
three-dimensional periodic Maxwell equations: uniqueness of solution holds 
for all positive wave numbers if the material parameter is non-absorbing and satisfies 
suitable non-trapping and smoothness conditions. 
To prove the uniqueness result we derive a so-called Rellich identity for a 
solution to the homogeneous variational problem. The solution estimates resulting from this 
integral identity allow us to show that the homogeneous variational problem 
has only the trivial solution for \emph{all} positive wave numbers.
% and it also gives an bound on this solution in terms of the wave number (and 
% other parameters of the scattering problem).

Our analysis extends the approach in~\cite{Hadda2011} that was motivated by an existence 
and uniqueness proof for solutions to rough surface scattering problems via Rellich identities in~\cite{Chand2005}.
For scalar periodic problems, a related technique has been used in~\cite{Bonne1994}.
The paper~\cite{Hadda2011} studied electromagnetic scattering from rough, unbounded 
penetrable layers. Such scattering problems are considered to be more complicated 
than those for periodic structures since the problem to find the scattered field 
cannot reduced, e.g., to a bounded domain. The applications of rough scattering problems 
include for instance outdoor noise propagation, oceanography or even optical technologies 
when the dielectric lacks periodicity.  
The authors in~\cite{Hadda2011} formulated the latter scattering problem in terms of the 
electric field. We will instead choose a formulation in terms of the magnetic field, which somewhat 
changes the role of the dielectric material parameter in the integral identities since the material is non-magnetic. 
The paper~\cite{Hadda2011}  establishes existence and uniqueness of solution under non-trapping 
and smoothness conditions on the material parameter. While a priori estimates 
resulting from the Rellich identity allowed the authors in~\cite{Hadda2011} to deduce uniqueness 
of solution, existence of solution has been obtained using a limiting absorption argument.
The approach studied in the present paper is, from the technical point of view, somewhat similar 
to the one introduced in~\cite{Hadda2011}. However, the analysis of the biperiodic case is 
definitely simpler since uniqueness of solution directly implies existence. Therefore, one only
needs to investigate the Rellich identity and estimates for solutions to the homogeneous problem.
It turns out also that this procedure produces weaker assumptions on the material parameter than 
those found in~\cite{Hadda2011}. More precisely, uniqueness and existence of
solution for all wave numbers are obtained under the following (non-trapping and smoothness)
assumptions on the biperiodic relative material parameter $\epsr: \, \R^3_+ := \{ x \in \R^3, \, x_3 > 0 \} \to \R$.
First, we assume that $\epsr^{-1} \in L^{\infty}(\R^3_+)$ equals one in $\{ x_3 > h \}$ for some $h>0$ 
and possesses essentially bounded and measurable first weak derivatives. Second, we require that 
\begin{align*}
& (a)\quad  \frac{\pa\epsr^{-1}}{\pa x_3} \leq 0 \text{ in } \R^3_+,\\
& (b) \quad \text{It holds that } \frac{\pa\epsr^{-1}}{\pa x_3} < 0 \text{ in some non-empty open subset of } \R^3_+,\\
& (c) \quad\text{There exists } \delta>1/2 \text{ such that } \frac{\delta}{2} \|\naT \epsr^{-1}\|_{L^{\infty}(\R^3_+)^3}^2 + \frac{\sqrt{2}}{h}\bigg\| \frac{\pa \epsr^{-1}}{\pa x_3} \bigg\|_{L^\infty(\R^3_+)} < \frac{2}{h^2},
\end{align*}
where $\nabla_T \epsr^{-1} := (\pa \epsr^{-1}/\pa x_1,\pa \epsr^{-1}/\pa x_2,0)^{\top}$. 
Under  these conditions, the existence of surface waves is  automatically ruled out. 
While conditions (a) and (c) are similar to conditions (a) and (d) in~\cite[Eq.~(7.2)]{Hadda2011}, 
condition (b) is weaker and clearly simpler than the corresponding conditions (b) and (c) in~\cite[Eq.~(7.2)]{Hadda2011}.
% In~\cite{Hadda2011}, one can find the conditions corresponding to (a) and (c)
% but the corresponding assumption on ``strictly decreasing'' 
% in (b) above, is roughly speaking required on some tubular domain which
% is not necessarily needed here.  
%   
% In contrast to~\cite{Hadda2011} we only need to establish uniqueness of solution 
% (but not existence of solution).
%}

The half-space setting that we consider  in this paper is somewhat special, and it 
seems worth to mention that the Rellich identity itself generalizes to a corresponding 
periodic scattering problem in full space. The resulting estimate for a solution 
$H$ to the scattering problem has a similar structure to the estimate in 
Lemma~\ref{th:estimatelemma}. However, in the half-space setting, the term 
$2\Re\int_{\Omega} (\pa \epsr^{-1} / \pa x_3) \, (\pa H_3 / \pa x_3)\ol{H_3}  \d{x}  $
can be treated without integration by parts using a Poincar\'e lemma. In contrast,
in the full-space setting the only obvious way of treating this term is to 
integrate by parts. Since we seek for solution estimates, this introduces the 
condition that $x_3 \mapsto \epsr^{-1}(x_1, x_2, x_3)$ needs to be concave 
to conclude. Since this is a somewhat unnatural condition, we do not present 
this result in more detail. 

One can further generalize the results presented here to certain anisotropic structures. 
However, already for the simpler case of isotropic coefficients the derivation of the Rellich identity 
is a technical matter. Again, we have opted to try to keep the presentation simple instead 
of treating the most general setting that could be considered. 

The paper is organized as follows: 
In Section 2 we present setting of the problem. 
Section 3 is dedicated to a variational formulation and to the Fredholm property of the latter. 
Section 4 contains a couple of technical lemmas. 
We derive the integral inequalities resulting from the Rellich identity in Section 5. 
Finally, the uniqueness of the variational problem for all wave numbers is proven in Section 6.

Notation: We denote by $H^s(\R^d)^3$, $d=2,3$, the usual $L^2$-based Sobolev 
space of vector-valued functions in $\R^d$. Moreover, 
$H^s_{\text{loc}}(\R^3)^3 = \{v \in H^s(B)^3 \text{ for all balls } B \subset \R^3\}$, 
and $W^{1,\infty}(\R^3) = \{ v \in L^\infty(\R^3): \, \nabla v \in L^\infty(\R^3)^3 \}$.

\section{Problem Setting}
We consider scattering of time-harmonic electromagnetic waves from a biperiodic structure which models a 
dielectric layer mounted on a perfectly conducting plate. The electric field $E$ and the magnetic field $H$ 
are governed by the time-harmonic Maxwell equations at frequency $\omega>0$ in $\R^3_+= \{(x_1,x_2,x_3) \in \R^3: x_3 >0 \}$,
\begin{align}
\label{eq:maxwellEquation1}
 \curl H + i\omega\varepsilon E & = 0\quad \text{ in } \R^3_+,\\
\label{eq:maxwellEquation2}
 \quad \curl E - i\omega\mu H & = 0\quad \text{ in } \R^3_+,\\
\label{eq:maxwellEquation3}
 e_3\times E & = 0 \quad \text{ on } \{x_3 = 0\}, 
%\label{eq:maxwellEquation4}
% e_3 \cdot H & = 0 \quad \text{ on } \{x_3 = 0\},
\end{align}
where $e_3 = (0,0,1)^\top$.  
The electric permittivity $\varepsilon$ is a bounded 
measurable function that is $2\pi$-periodic in $x_1$ and $x_2$. Further, we 
assume that $\varepsilon$ equals $\varepsilon_0>0$ outside the biperiodic structure, 
that is, for $x_3 \geq h$ where $h>0$ is chosen larger than $\sup \{ x_3: \, (x_1,x_2,x_3)^\top \in \supp(\varepsilon -\varepsilon_0) \}$.
The magnetic permeability $\mu=\mu_0$ is assumed to be a positive constant 
and the conductivity  is assumed to vanish. As usual, the 
problem~\eqref{eq:maxwellEquation1}-\eqref{eq:maxwellEquation3} 
has to be completed by a radiation condition that we set up using Fourier series. 

The biperiodic structure is illuminated by an electromagnetic plane wave with wave vector 
$d = (d_1,d_2,d_3) \in \R^3$, $d_3 <0$, such that  $d\cdot d = \omega^2\varepsilon_0 \mu_0$. 
The polarizations $p,q \in \R^3$ of the incident wave satisfy $p\cdot d = 0$ and 
$q = 1 / (\omega\varepsilon_0) (p\times d)$. With these definitions, the incident plane 
waves $E^i$ and $H^i$ are given by 
\[
 E^i := q e^{id\cdot x} , \quad H^i := p e^{id\cdot x}, \quad x \in \R^3_+.
\]
In the following we will exploit that one can explicitly compute the 
corresponding reflected field at $\{ x_3 = 0 \}$. To this end, we introduce 
the notation $\tilde{a} = (a_1, a_2, - a_3)^\top$ for $a=(a_1,a_2,a_3)^\top \in \R^3$. 
The reflected waves at the plane $\{ x_3 = 0 \}$ are 
\[
  % E^i = s e^{id\cdot x} , \quad H^i = p e^{id\cdot x}, \quad x \in \R^3_+. 
  E^r(x) := - \tilde{q} e^{i \tilde{d} \cdot x}, \quad 
  H^r(x) :=  \tilde{p} e^{i \tilde{d} \cdot x}, \quad x \in \R^3_+,  
\]
since $\div E^r = 0$, $\div H^r = 0$, and 
$e_3 \times (E^i + E^r) = 0$, $e_3 \cdot (H^i + H^r) = 0$ 
on $\{ x_3 = 0\}$. From now on, we denote the sum of the incident 
and reflected plane waves by 
\[
  E^{ir} :=  E^i + E^r \quad \text{and} \quad 
  H^{ir} := H^i + H^r.
\]
Set 
\[
  \alpha = (\alpha_1,\alpha_2,\alpha_3)^\top := (d_1,d_2,0)^\top
\]
and define $E^{ir}_{\alpha}$ and $H^{ir}_{\alpha}$ by
\[
  E^{ir}_{\alpha} := e^{-i \alpha \cdot x}E^{ir}(x), \quad 
  H^{ir}_{\alpha} := e^{-i \alpha \cdot x} H^{ir}(x), \quad x \in \R^3_+,
\]
such that $E^{ir}_{\alpha}$ and $H^{ir}_{\alpha}$ are $2\pi$-periodic 
in $x_1$ and $x_2$. If we apply the same phase shift to solutions 
$E$ and $H$ of the Maxwell
equations~\eqref{eq:maxwellEquation1}-\eqref{eq:maxwellEquation3},
\[
 E_{\alpha} = e^{-i \alpha \cdot x}E(x), \quad H_{\alpha} = e^{-i \alpha \cdot x}H(x),
\]
and if we denote  
\[
% \frac{\pa^{\alpha}}{\pa x_j} = \frac{\pa}{\pa x_j} + i\alpha, \quad j=1,2,3, \\
\nabla_\alpha f = \nabla f + i \alpha f, \quad 
\curl_{\alpha} F = \curl F + i\alpha \times F, \quad 
\div_{\alpha}F = \div F + i\alpha \cdot F
\]
for scalar functions $f$ and vector fields $F$, 
then $E_{\alpha}$ and $H_{\alpha}$ satisfy
\begin{align}
\label{eq:PmaxwellEquation1}
 \curl_{\alpha} H_{\alpha} + i\omega \varepsilon E_{\alpha} & = 0 \quad \text{ in } \R^3_+,\\
\label{eq:PmaxwellEquation2}
 \curl_{\alpha} E_{\alpha} - i\omega\mu_0 H_{\alpha} & = 0 \quad \text{ in } \R^3_+, \\
\label{eq:PmaxwellEquation3}
 e_3 \times E_{\alpha} & = 0 \quad \text{ on } \{x_3 = 0\}.
%\label{eq:PmaxwellEquation4}
% e_3 \cdot H_{\alpha} & = 0 \quad \text{ on } \{x_3 = 0\},
\end{align}
Note that we still have $\div_{\alpha}\curl_{\alpha} = 0$ and 
$\curl_{\alpha}\nal = 0$. Let us denote the relative material parameter by 
\[
  \epsr:=\frac{\varepsilon}{\varepsilon_0}. 
\]
Obviously, $\epsr$ equals one outside the biperiodic dielectric structure. 
Recall that the magnetic permeability $\mu_0$ is constant which motivates
us to work with the divergence-free magnetic field, that is, $\div_{\alpha}H_{\alpha} = 0$. 

Note that~\eqref{eq:PmaxwellEquation1} plugged in 
into~\eqref{eq:PmaxwellEquation3} implies that 
$e_3 \times (\epsr^{-1} \curl_\alpha H_\alpha) = 0$ 
on $\{x_3 = 0 \}$ and that the condition $e_3 \cdot H_{\alpha} = 0$ on $\{x_3 = 0 \}$
can be derived by plugging~\eqref{eq:PmaxwellEquation3} into~\eqref{eq:PmaxwellEquation2}.
Hence, introducing the wave number 
$k= \omega (\epsilon_0 \mu_0)^{1/2}$, 
and eliminating the electric field $E_{\alpha}$ 
from~\eqref{eq:PmaxwellEquation1}-\eqref{eq:PmaxwellEquation3}, 
we find that 
\begin{align}
\label{eq:Order2Total}
  \curl_{\alpha} \big(\epsr^{-1} \curl_{\alpha} H_{\alpha}\big) - k^2  H_{\alpha} & = 0\quad \text{ in } \R^3_+, \\
  % e_3 \times (\epsr^{-1} \curl_{\alpha} H_{\alpha}) & = 0 \quad \text{ on } \{x_3 = 0\}
  e_3  \times (\epsr^{-1} \curl_\alpha H_{\alpha}) & = 0 \quad \text{ on } \{x_3 = 0\}, \\
  e_3 \cdot H_{\alpha} & = 0 \quad \text{ on } \{x_3 = 0\}.
\end{align}
We wish to reformulate the last three equations in terms of the 
scattered field $H^s_{\alpha}$, defined by 
$H^s_{\alpha}:= H_{\alpha} - H^{ir}_{\alpha}$. 
Since, by construction, 
$\curl_{\alpha}\curl_{\alpha} H^{ir}_{\alpha} - k^2  H^{ir}_{\alpha} = 0$  
in  $\R^3_+$, $H^{ir}_\alpha \cdot e_3 = 0$ and 
$e_3 \times (\epsr^{-1} \curl_\alpha H^{ir}_\alpha) = 0$ on 
$\{ x_3 = 0 \}$, a simple computation shows that 
\begin{equation}
\label{eq:secondOrder}
\begin{split}
 \curl_{\alpha} \big(\epsr^{-1} \curl_{\alpha} H^s_{\alpha}\big) - k^2  H^s_{\alpha} & 
   = -\curl_{\alpha} \big( (\epsr^{-1}-1) \curl_{\alpha} H^{ir}_{\alpha}\big)\quad \text{ in } \R^3_+, \\
 % e_3 \times (\epsr^{-1} \curl_{\alpha} H^s_{\alpha}) & 
 %  = -e_3 \times (\epsr^{-1} \curl_{\alpha} H^i_{\alpha}) \quad \text{ on } \{x_3 = 0\}.
 e_3  \times (\epsr^{-1} \curl_\alpha H_{\alpha}^s) & = 0 \quad \text{ on } \{x_3 = 0\}, \\
 e_3 \cdot H_{\alpha}^s & = 0  \quad \text{ on } \{x_3 = 0 \} .
\end{split}
\end{equation}
Due to the biperiodicity of the right-hand side and of $\epsr$, 
we seek for a biperiodic solution $H^s_\alpha$, and reduce the 
problem to the domain $(0,2\pi)^2\times(0,\infty)$. 
%
% \[
%  \curl_{\alpha} H^s_{\alpha} + i\omega\varepsilon_0 \epsr E^s_{\alpha} = -i\omega\varepsilon_0(\epsr - 1)E^i_{\alpha}, \quad  \curl_{\alpha} E^s_{\alpha} - i\omega\mu_0 H^s_{\alpha} = 0,\, E^s_{\alpha}\times e_3=0 \text{ on } \Gamma_0.
% \]
% We eliminate the electric field $E^s_{\alpha}$ from the second equation of this system and obtain 
% \begin{eqnarray}
%  \curl_{\alpha} \big(\epsr^{-1} \curl_{\alpha} H^s_{\alpha}\big) - k^2  H^s_{\alpha} = - i\omega\varepsilon_0\curl_{\alpha}\big((\varepsilon - \varepsilon_0)E^i_{\alpha}\big), \\
%  e_3 \times \curl_{\alpha} H^s_{\alpha} = 0 \text{ on } \Gamma_0.
%% \end{eqnarray}To sum up, 
%the boundary value problem under investigation is 
%\begin{eqnarray}
%\label{eq:secondOrder}
% \curl_{\alpha} \big(\epsr^{-1} \curl_{\alpha} H\big) - k^2  H = \curl_{\alpha}(q G) 
%   \quad \text{ in }(0,2\pi)^2\times(0,\infty), \\
%\label{eq:BC0}
%  e_3 \times ( \epsr^{-1}  \curl_{\alpha} H ) = e_3 \times (\epsr^{-1} G)
%  \quad \text{ on } (0,2\pi)^2\times\{0\}.
%\end{eqnarray}
%where $q:=\epsr^{-1}-1$ is the contrast and the $G:=- \curl_{\alpha} H^i_{\alpha}$ 
%is source function. 
We complement this boundary value problem by a radiation condition, 
see also in~\cite{Dobso1994,Bao2000}, that we set up using Fourier series. 
The scattered field $H^s_\alpha$ is $2\pi$-periodic 
in $x_1$ and $x_2$ and can hence be expanded as 
\begin{equation}
\label{eq:FourierH}
  H^s_\alpha(x) = \sum_{n \in \Lambda} \hat{H}_n(x_3) e^{i n \cdot x}, 
  \quad x = (x_1,x_2,x_3)^\top \in \R^3_+, \ \Lambda=\Z^2 \times \{0\},
\end{equation}
where the Fourier coefficients $\hat{H}_n(x_3)$ are defined by
\begin{equation}
\label{eq:Fourier}
 \hat{H}_n(x_3)= \frac{1}{4 \pi^2} \int_{0}^{2\pi}\int_{0}^{2\pi} H^s_\alpha(x_1,x_2,x_3) e^{-i n \cdot x} \d{x_1} \d{x_2}, 
 \quad n \in \Lambda. 
\end{equation}
Define 
\[
  \beta_n:= \begin{cases} 
              \sqrt{k^2 - |n+\alpha|^2}, & k^2 \geq |n+\alpha|^2, \\ 
              i \sqrt{|n+\alpha|^2- k^2}, & k^2 < |n+\alpha|^2,
            \end{cases}
  \quad n \in \Lambda.
\] 
Since $\epsr^{-1}$ equals one for $x_3>h$ it holds that 
$\div_{\alpha} H^s_\alpha$ vanishes for $x_3>h$,
and equation~\eqref{eq:secondOrder} becomes
$(\Delta_{\alpha} + k^2)  H^s_\alpha = 0$ in $\{ x_3 >h \}$, 
where $\Delta_{\alpha} = \Delta + 2i \alpha\cdot\nabla - |\alpha|^2$. 
Using separation of variables, and choosing the upward propagating 
solution, we set up a radiation condition in form of a Rayleigh 
expansion condition, prescribing that $H^s_\alpha$ can be written as 
\begin{equation}
\label{eq:radiationCondition}
  H^s_\alpha(x) = \sum_{n\in \Lambda} \hat{H}_n e^{i \beta_n (x_3-h) + in\cdot x}
  \quad \text{ for } \{ x_3 >h \}, \quad 
  \text{where }\hat{H}_n := \hat{H}_n(h),
\end{equation}
and that the series converges uniformly in compact subsets of 
$\{ x_3 >h \}$. 

The scattering problem to find a scattered field $H^s_{\alpha}$ 
that satisfies the boundary value problem~\eqref{eq:secondOrder} 
and the expansion~\eqref{eq:radiationCondition} is in the 
following section reformulated variationally in a suitable Sobolev 
space. 
%To simplify the notation, we simply write $H$ instead of 
%$H^s_{\alpha}$ in the rest of the text. 

\section{Variational Formulation}

We solve the scattering problem presented in the last section 
variationally, and briefly recall in this section a variational 
formulation of the problem in a suitable Sobolev space. Our framework 
is an adaption of the results from~\cite{Schmi2003} to our half-space setting.
In contrast to the variational formulation in $H(\curl)$ 
in~\cite{Abbou1993}, the papers~\cite{Dobso1994, Bao1997, Bao2000, Schmi2003}
set up a variational formulation in $H^1$ for the magnetic field. Indeed, 
since the latter is divergence-free, any solution that is locally $H(\curl)$ indeed 
belongs locally to $H^1$. For our purposes, the $H^1$ formulation has the 
additional advantage that it is well-defined at Rayleigh-Wood frequencies, 
as it was noted in~\cite{Schmi2003}. 
\begin{figure}%[hbt]
\centering
\psfrag{x3h}{}
\psfrag{x30}{}
\psfrag{0}{0}
\psfrag{2pi}{$2\pi$}
\psfrag{B2}{$\Gamma_0$}
\psfrag{D}{}
\psfrag{x3}{$x_3$}
\psfrag{x1x2}{$(x_1,x_2)$}
\psfrag{gammah}{$\Gamma_h$}
\psfrag{omega}{$\Omega$}
\includegraphics[width=8cm]{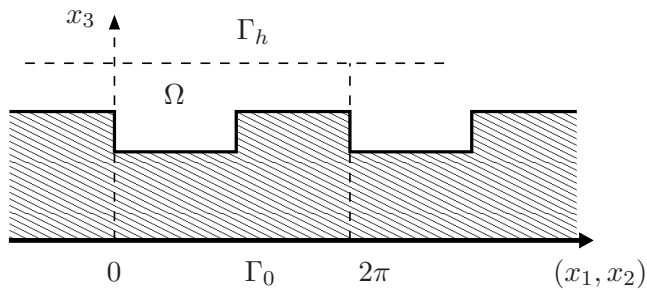}
\caption{Geometric setting for electromagnetic scattering problem from a biperiodic dielectric 
structure mounted on a perfectly conducting plate (in two dimensions, for simplicity).}
\end{figure}
We define a bounded domain
\[
 \Omega=(0,2\pi)^2\times(0,h) \qquad 
 \text{for } h > \sup \{ x_3 : \, (x_1,x_2,x_3)^\top \in \supp(\epsr-1) \},
\]
with  boundaries $\Gamma_0 := (0,2\pi)^2\times\{0\}$  and 
$\Gamma_h := (0,2\pi)^2\times\{h\}$, and Sobolev spaces 
\begin{align*}
  H^\ell_{\mathrm{p}}(\Omega)^3 & := \{ F \in H^\ell(\Omega)^3: F = \tilde{F}|_{\Omega}\text{ for some } 2\pi \text{-biperiodic } \tilde{F}\in H^\ell_{\text{loc}}(\R^3)^3\}, \quad \ell \in \N,\\
  H^1_{\mathrm{p}, \mathrm{T}}(\Omega)^3 & := \{ F=(F_1,F_2,F_3)^\top \in H^1_{\mathrm{p}}(\Omega)^3: \, F_3 = 0 \text{ on } \Gamma_0 \},
\end{align*}
equipped with the usual integral norm, e.g.,
\[
  \| F \|^2_{H^1_{\mathrm{p}}(\Omega)^3} = \| F \|^2_{L^2(\Omega)^3} + \| \nabla_\alpha F \|^2_{L^2(\Omega)^3}.
\]
The space $H^1_{\mathrm{p}, \mathrm{T}}(\Omega)^3$ of periodic vector fields that are tangential 
on $\Gamma_0$ is well-defined due to the standard trace theorem in $H^1$.
We also define periodic Sobolev spaces of functions with $d=1,2,3$ components 
on $\Gamma_h$: for $s \in \R$, 
\[
  H^s_{\mathrm{p}}(\Gamma_h)^d 
  := \{ F \in H^s(\Gamma_h)^d: \, F = \tilde{F}|_{\Gamma_h} \text{ for some } 2\pi \text{-biperiodic } \tilde{F} \in H^s_{\text{loc}}( \{ x_3 = h \})^d\}.
\]
A periodic vector field $F \in H^s(\Gamma_h)^d$ can be developed in a 
Fourier series, $F(x) = \sum_{n \in \Lambda} \hat{F}_n \exp(i n \cdot x)$, 
and $\| F \|_{H^s_{\mathrm{p}}(\Gamma_h)^d} 
   = ( \sum_{n \in \Lambda} (1+ n^2)^{s} |\hat{F}_n|^2 )^{1/2}$
defines a norm on $H^s_{\mathrm{p}}(\Gamma_h)^d$. 

We define a non-local boundary operator $T_\alpha$ (the exterior Dirichlet-Neumann operator) by
\[
  (T_\alpha f)(x) = \sum_{n \in \Lambda} i \beta_n \hat{f}_n e^{i n \cdot x},
  \quad \text{for }
  f  = \sum_{n \in \Lambda} \hat{f}_n \exp(i n \cdot x) \in H^{1/2}_{\mathrm{p}}(\Gamma_h).
\]
It is a classical result that $T_\alpha $ is bounded from $H_{\mathrm{p}}^{{1/2}}(\Gamma_h)$ 
into $H_{\mathrm{p}}^{-{1/2}}(\Gamma_h)$, see, e.g.,~\cite{Arens2010}.
% \begin{proof}
%   For $f \in H_{\mathrm{p}}^{{1/2}}(\Gamma_h)$, 
% %the norm $\|f\|_{H_{\mathrm{p}}^{{1/2}}(\Gamma_h)} $ is given by 
% % \[
% %  \|f\|^2_{H_{\mathrm{p}}^{{1/2}}(\Gamma_h)} = \sum_{n\in \Lambda} (1+|n|^2)^{{1/2}} |\hat{f}_n|^2.
% % \]
% we have 
% \begin{eqnarray*}
% \|T_\alpha f\|^2_{H_{\mathrm{p}}^{-{1/2}}(\Gamma_h)} &=& \sum_{n\in \Lambda} (1+|n|^2)^{-{1/2}} |\beta_n \hat{f}_n|^2 \\ &\leq& C\sum_{n\in \Lambda} (1+|n|^2)^{{1/2}} |\hat{f}_n|^2 \leq C\|f\|^2_{H_{\mathrm{p}}^{{1/2}}(\Gamma_h)}
% \end{eqnarray*}
%  which implies the boundedness of $T_\alpha $.
% \end{proof}
Using $T_\alpha$, we define a vector of (pseudo-)differential 
operators  $R_\alpha :=  (\partial^{\alpha}/\pa x_1, \partial^{\alpha}/\pa x_2, T_\alpha )$.  
For a vector field $F \in H_{\mathrm{p}}^{{1/2}}(\Gamma_h)^3$,   
\[
  R_\alpha \times F =  (\partial^{\alpha}/\pa x_1, \partial^{\alpha}/\pa x_2, T_\alpha ) \times F, 
  \quad 
  R_\alpha \cdot F =  (\partial^{\alpha}/\pa x_1, \partial^{\alpha}/\pa x_2, T_\alpha ) \cdot F.
\]
Since all components of $R_\alpha$ are bounded operators from $H_{\mathrm{p}}^{{1/2}}(\Gamma_h)$ 
into $H_{\mathrm{p}}^{-{1/2}}(\Gamma_h)$,  the operator $F \mapsto R_\alpha \times F$ is 
bounded from $H^{1/2}_{\mathrm{p}}(\Gamma_h)^3$ into $H^{-1/2}_{\mathrm{p}}(\Gamma_h)^3$, and 
$F \mapsto R_\alpha \cdot F$ is bounded from $H^{1/2}_{\mathrm{p}}(\Gamma_h)^3$ into 
$H^{-1/2}_{\mathrm{p}}(\Gamma_h)$. If a biperiodic function $H \in H^1_{\mathrm{loc}}(\R^3_+)$ satisfies 
the Rayleigh expansion condition, then $T_\alpha  H_3 = \partial H_3 / \partial x_3$ on $\Gamma_h$. 
This implies that $e_3 \times (\curl_\alpha H) = e_3 \times (R_\alpha \times H)$ on $\Gamma_h$ 
(see, e.g.,~\cite{Schmi2003}).

Assume that $H^s_\alpha$ is a distributional periodic solution 
to the boundary value problem~\eqref{eq:secondOrder} such that 
$H^s_\alpha$, $\curl_\alpha H^s_\alpha$, and $\div_\alpha H^s_\alpha$
are locally square-integrable, such that the radiation 
condition~\eqref{eq:radiationCondition} is satisfied, 
and such that $\nu \cdot (H^s_\alpha+H^{ir}_\alpha)$ and 
$\nu \times (\epsr^{-1} \curl (H^s_\alpha+H^{ir}_\alpha))$ are 
continuous over interfaces with normal vector $\nu$ where $\epsr$ jumps. 
As noted in~\cite{Schmi2003}, this implies that, following the above 
notation, $H^s_\alpha \in H^1_{\mathrm{p},\mathrm{T}}(\Omega)$.
Then the Stokes formula~\cite{Abbou1993, Schmi2003} implies that 
\begin{align*}
  & \int_{\Omega}(\epsr^{-1}\curl_{\alpha} H^s_\alpha \cdot \overline{\curl_{\alpha} F} 
    - k^2  H^s_\alpha \cdot \overline{F} )\d{x} \\
  & - \int_{\Gamma_0} e_3 \times (\epsr^{-1} \curl_\alpha H^s_\alpha) \cdot \overline{F}\d{x} 
  + \int_{\Gamma_h} e_3 \times (R_\alpha\times H^s_\alpha) \cdot \overline{F}\d{s} \\ 
  & = \int_{\Omega} (1-\epsr^{-1}) \curl_\alpha H^{ir}_\alpha \cdot \overline{\curl_{\alpha} F}\d{x} 
  - \int_{\Gamma_0} (e_3 \times (1-\epsr^{-1}) \curl_\alpha H^{ir}_\alpha) \cdot \overline{F}\d{x} 
\end{align*}  
for all test functions $F \in H^1_{\mathrm{p}, \mathrm{T}}(\Omega)^3$. 
Since we assumed that 
\[
  0 = e_3 \times (\epsr^{-1} \curl_\alpha H_\alpha)
  = e_3 \times (\epsr^{-1} \curl_\alpha (H^s_\alpha+H^{ir}_\alpha)) \quad \text{on } \Gamma_0,
\]
the above identity simplifies to 
\begin{multline*}
  \int_{\Omega}(\epsr^{-1}\curl_{\alpha} H^s_\alpha \cdot \overline{\curl_{\alpha} F} 
    - k^2 H^s_\alpha \cdot \overline{F} )\d{x}
  + \int_{\Gamma_h} e_3 \times (R_\alpha\times H^s_\alpha) \cdot \overline{F}\d{s}  \\
  = \int_{\Omega} (1-\epsr^{-1}) \curl_\alpha H^{ir}_\alpha \cdot \overline{\curl_{\alpha} F}\d{x} 
  - \int_{\Gamma_0} (e_3 \times \curl_\alpha H^{ir}_\alpha) \cdot \overline{F}\d{x}.
\end{multline*}  
% For the above-introduced special fields $H^{ir}_\alpha$ the tangential
% component of $\curl_\alpha H^{ir}_\alpha$ vanishes, but will later on replace 
% $\curl H^{ir}_\alpha$ by a general source function. To this end, we do
% not neglect this term.
By construction, $e_3 \times \curl_\alpha H^{ir}_\alpha$ vanishes on $\Gamma_0$, 
that is, we can neglect the last term in the last equation.
The divergence constraint $\div_\alpha H^s_\alpha = 0$ that follows 
from~\eqref{eq:secondOrder} shows that
\begin{align}
  \mathcal{B}(H^s_\alpha,F) & := 
  \int_{\Omega}(\epsr^{-1}\curl_{\alpha} H^s_\alpha \cdot \overline{\curl_{\alpha} F} 
                - k^2 H^s_\alpha \cdot \overline{F} )\d{x}
  + \rho\int_{\Omega} (\div_{\alpha} H^s_\alpha)(\overline{\div_{\alpha} F})\d{x} \nonumber \\ 
  & \quad + \int_{\Gamma_h} e_3 \times (R_\alpha\times H^s_\alpha) \cdot \overline{F}\d{s} 
    - \int_{\Gamma_h} (R_\alpha\cdot H^s_\alpha) (e_3 \cdot \overline{F}) \d{s} \nonumber \\
  & = \int_{\Omega} (1-\epsr^{-1}) \curl_\alpha H^{ir}_\alpha \cdot \overline{\curl_{\alpha} F}\d{x},   \label{eq:prepVariational}
  %+ \int_{\Gamma_0} (e_3 \times \curl_\alpha H^{ir}_\alpha) \cdot \overline{F}\d{x} 
\end{align}  
where $\rho$ is some complex constant with $\Re(\rho) \geq c >0$ and $\Im(\rho)<0$.  

%\textcolor{blue}{ 
We next prove that the bounded sesquilinear form 
$\mathcal{B}: \, H^1_{\mathrm{p}, \mathrm{T}}(\Omega)^3 \times H^1_{\mathrm{p}, \mathrm{T}}(\Omega)^3 \to \C$ 
satisfies a G\r{a}rding inequality (this goes back to~\cite{Abbou1993}), i.e. there exist strictly positive 
constants $c_1$ and $c_2$ such that
\begin{align}
\label{eq:Garding}
\Re(\mathcal{B}(H,H)) \geq c_1\int_{\Omega}|\nabla_{\alpha}H|^2\d{x}  - c_2\int_\Omega |H|^2\d x.
\end{align}
for all $H\in H^1_{\mathrm{p}, \mathrm{T}}(\Omega)^3$.%}
\begin{theorem}
 \label{th:Fredholm}
Assume that $\epsr^{-1} \in L^\infty(\Omega)$ is positive and bounded away 
from zero. Set $\Re \rho = \inf_{\Omega} \epsr^{-1} >0$ and choose 
$\Im \rho <0$. Then $\mathcal{B}$ satisfies~\eqref{eq:Garding}. 
\end{theorem} 

\begin{proof}
%  Denote $\mathcal{B}_1(H,F) = \mathcal{B}(H,F) + k^2 \int_{\Omega}H \cdot \overline{F}$ 
%  for $H, F \in H^1_{\mathrm{p}, \mathrm{T}}(\Omega)^3$. 
As in~\cite[proof of Theorem 1]{Schmi2003} one shows that 
 \begin{multline*}
   \Re(\mathcal{B}(H,H)) \geq \Re (\rho) \int_{\Omega}(|\curl_{\alpha} H|^2 + |\div_{\alpha} H|^2)\d{x}
   - k^2 \int_{\Omega} |H|^2 \d{x} \\ 
   - \Re \int_{\Gamma_h} T_\alpha H\cdot \overline{H}\d{s} 
   - 2 \Re\int_{\Gamma_h} \left(\overline{H_3}\frac{\pa^{\alpha}H_1}{\pa x_1} 
     + \overline{H_3}\frac{\pa^{\alpha}H_2}{\pa x_2} \right) \d{s}.
\end{multline*}
The following identity follows from integrations by parts, the periodicity, and the vanishing 
normal component of $H$ on $\Gamma_0$, 
\[
  \int_{\Omega}(|\curl_{\alpha} H|^2 + |\div_{\alpha} H|^2)\d{x} 
  = \int_{\Omega}|\nabla_{\alpha}H|^2\d{x} 
  + 2 \Re \int_{\Gamma_h} \bigg(\overline{H_3}\frac{\pa^{\alpha}H_1}{\pa x_1} 
  + \overline{H_3}\frac{\pa^{\alpha}H_2}{\pa x_2}\bigg)\d{s}.
\]
In consequence, 
\begin{multline*}
  \Re(\mathcal{B}(H,H)) \geq \Re (\rho) \int_{\Omega}|\nabla_{\alpha}H|^2\d{x} 
  - k^2 \int_{\Omega} |H|^2 \d{x} \\ 
  - \Re\int_{\Gamma_h} T_\alpha H\cdot \overline{H}\d{s} 
  - 2(1 - \Re (\rho)) \,  \Re \int_{\Gamma_h} \bigg(\frac{\pa^{\alpha}H_1}{\pa x_1} + \frac{\pa^{\alpha}H_2}{\pa x_2}\bigg)\overline{H_3}\d{s}. 
\end{multline*}
Precisely as in~\cite{Schmi2003} one shows now by a Fourier series argument that   
%\textcolor{blue}{
\begin{align*}
 -\Re\int_{\Gamma_h} T_\alpha H\cdot \overline{H}\d{s} 
- 2 (1 - \Re (\rho))\Re \int_{\Gamma_h} \bigg(\frac{\pa^{\alpha}H_1}{\pa x_1} + \frac{\pa^{\alpha}H_2}{\pa x_2}\bigg)\overline{H_3}\d{s}
 &\geq \Re\int_{\Gamma_h} K(H) \cdot \ol{H} \d{s} \\ &\geq -C\int_\Omega |H|^2\d x
\end{align*}
for a finite-dimensional operator $K$ on $H^{1/2}_{\mathrm{p}}(\Gamma_h)^3$. Note that
the last inequality follows from $|\int_{\Gamma_h} K(H)\cdot\overline{H}\d s|\leq C\int_\Omega |H|^2\d x$ 
due to the finite-dimensional range of $K$ and the fact that on finite-dimensional spaces all norms are equivalent. 
The last inequality implies a G\r{a}rding inequality for $\mathcal{B}$. %}
\end{proof}

For simplicity we write from now on $H$ for the searched-for scattered field 
$H^s_\alpha$ in~\eqref{eq:prepVariational} and replace the source function 
$\curl H^{ir}_\alpha$ by a $G \in H_{\mathrm{p}}^1(\Omega)^3$. 
%\textcolor{blue}{
The last theorem implies the following corollary. 
\begin{corollary}
The variational problem to find 
$H \in H^1_{\mathrm{p}, \mathrm{T}}(\Omega)^3$ such that 
\begin{equation}
  \label{eq:weakProblem}
  \mathcal{B}(H,F) 
  = \int_{\Omega} (1-\epsr^{-1}) G \cdot \overline{\curl_{\alpha} F}\d{x} 
  \quad \text{for all } F \in H^1_{\mathrm{p}, \mathrm{T}}(\Omega)^3 
  % + \int_{\Gamma_0} (e_3 \times G) \cdot \overline{F}\d{x}
\end{equation}
satisfies the Fredholm alternative, i.e., uniqueness of solution implies 
existence of solution.
\end{corollary}%}

Note that this formulation corresponds to the 
usual variational formulation of the Maxwell equations with perfectly 
conducting magnetic boundary conditions in smooth bounded domains, 
see, e.g.,~\cite[Section 4.5(b)]{COSTABEL:2010:HAL-00453934:2}.
For special material parameters $\epsr^{-1}$ in
\[
  W^{1,\infty}_{\mathrm{p}}(\Omega) := \{ f \in L^\infty(\Omega): \, 
  f = \tilde{f}|_{\Omega} \text{ for some $2\pi$-biperiodic } \tilde{f} \in W^{1,\infty}(\R^3) \}
\] 
we will in the sequel of the paper establish a uniqueness result via 
a Rellich identity. The next lemma will be useful when proving this
identity.

\begin{lemma}
\label{th:strongRelation1}
 Assume that $\epsr^{-1} \in W^{1,\infty}_{\mathrm{p}}(\Omega)$ is positive and bounded away from zero, and that 
 $G \in H^1_{\mathrm{p}}(\Omega)^3$. Then a 
solution $H \in H^1_{\mathrm{p}, \mathrm{T}}(\Omega)^3$ to problem~\eqref{eq:weakProblem} 
satisfies 
\begin{align}
\label{eq:eqOmega1}
 \curl_{\alpha} (\epsr^{-1} \curl_{\alpha} H) - k^2  H & = \curl_{\alpha}( (1-\epsr^{-1}) G) 
 \qquad \text{in } L^2(\Omega)^3, \\
 \div_\alpha H & = 0 \qquad \text{in } L^2(\Omega), \\
 e_3 \times (\epsr^{-1} \curl_\alpha H) & = e_3 \times ((1-\epsr^{-1}) G) \qquad \text{in } H^{-1/2}_{\mathrm{p}}(\Gamma_0)^3, \\
 e_3 \cdot H & = 0 \qquad \text{in } H^{1/2}_{\mathrm{p}}(\Gamma_0). \label{eq:eqOmega2}
\end{align}
Moreover, 
\begin{equation}
\label{eq:boundaryId}
  e_3  \times R_\alpha \times H = e_3 \times \curl_\alpha H 
  \quad \text{ in } H^{-1/2}_{\mathrm{p}}(\Gamma_h)^3 \text{ and} \quad  
  R_\alpha \cdot H = 0 \quad \text{ in } H^{-1/2}_{\mathrm{p}}(\Gamma_h),
\end{equation}
and $\partial H / \partial x_3 = T_\alpha (H)$ holds in $H^{-1/2}_{\mathrm{p}}(\Gamma_h)$.
\end{lemma}
\begin{proof}
The proof that $\div_{\alpha} H =0$ is analogous to the proof of~\cite[Theorem~2]{Schmi2003}.
In consequence, using a test function $F \in C^{\infty}_0(\Omega)^3$ in 
the variational problem~\eqref{eq:weakProblem} shows that the solution 
$H$ satisfies the differential equation~\eqref{eq:eqOmega1} in the 
distributional sense. 
Since $H \in H^1_{\mathrm{p}, \mathrm{T}}(\Omega)^3$,~\eqref{eq:eqOmega1} holds 
in the $L^2$-sense if the right-hand side belongs to $L^2(\Omega)^3$, which holds 
if $\epsr^{-1} \in W^{1,\infty}_{\mathrm{p}}(\Omega)$ and 
$G \in H^1_{\mathrm{p}}(\Omega)^3$.

Multiplying~\eqref{eq:eqOmega1} by 
$F \in H^1_{\mathrm{p},\mathrm{T}}(\Omega)^3$, using the Stokes formula, 
and subtracting the resulting expression from the variational 
formulation~\eqref{eq:weakProblem}, we find that 
\begin{multline*}
   \int_{\Gamma_h} e_3 \times (R_\alpha\times H) \cdot \overline{F}\d{s} 
   - \int_{\Gamma_h} (R_\alpha\cdot H) (e_3 \cdot \overline{F}) \d{s} 
   - \int_{\Gamma_h} e_3 \times \curl_\alpha H \cdot \ol{F} \d{s}\\
   + \int_{\Gamma_0} e_3 \times (\epsr^{-1} \curl_\alpha H) \cdot \ol{F} \d{s}
   - \int_{\Gamma_0} e_3 \times ((1-\epsr^{-1}) G) \cdot \ol{F} \d{s}
   = 0.
\end{multline*}
If we choose $F$ such that $\left. F \right|_{\Gamma_h} = 0$, then 
we see that $e_3 \times (\epsr^{-1} \curl_\alpha H - (1-\epsr^{-1}) G) =0$ in 
$H^{-1/2}_{\mathrm{p}}(\Gamma_0)$. If $\left. e_3 \cdot F \right|_{\Gamma_h} = 0$, 
it follows that $e_3 \times (R_\alpha\times H) = e_3 \times \curl_\alpha H$ in 
$H^{-1/2}_{\mathrm{p}}(\Gamma_h)^3$. Hence, $R_\alpha\cdot H = 0$ 
in $H^{-1/2}_{\mathrm{p}}(\Gamma_h)$.
These identities imply that $\partial H / \partial x_3 = T_\alpha (H)$ in 
$H^{-1/2}_{\mathrm{p}}(\Gamma_h)$ due to~\cite[Lemma 1]{Schmi2003}.
\end{proof}

\begin{remark}
Instead of the above variational formulation in 
$H^1_{\mathrm{p}, \mathrm{T}}(\Omega)^3$, one 
can also consider formulations in $H_{\mathrm{p}}(\curl_{\alpha},\Omega)^3$, 
the natural energy space for the second-order Maxwell
equations~\eqref{eq:secondOrder}, see, e.g.,~\cite{Abbou1993}. 
In $H_{\mathrm{p}}(\curl_{\alpha},\Omega)^3$  there is no bounded trace operator 
for the normal component of the field, and in consequence, the 
formulation~\eqref{eq:weakProblem} needs to be adapted. 
Usually, one replaces $F \mapsto e_3\times (R_\alpha \times F)\times e_3$ by $Q(e_3 \times H)$, where 
$Q$ is a bounded operator between the natural trace spaces 
$H^{-{1/2}}_{\mathrm{p},\div}(\Gamma_h)$ and $H^{-{1/2}}_{\mathrm{p},\curl}(\Gamma_h)$, 
 defined by
\begin{equation}
  (QF)(x) = -\sum_{n\in \Lambda}\frac{1}{i \beta_n}\{k^2\hat{F}_{T,n} - [(n+\alpha)\cdot\hat{F}_n](n+\alpha)\}e^{i n\cdot x},
  \quad \text{for }
  F(x) = \sum_{n\in \Lambda} \hat{F}_{n} e^{i n\cdot x},
\end{equation}
see, e.g.,~\cite{Abbou1993}. Obviously this definition only makes sense if all $\beta_n$ are 
non-zero. If this is the case, then the variational formulation~\eqref{eq:weakProblem} 
is equivalent to the formulation in $H_{\mathrm{p}}(\curl_{\alpha},\Omega)^3$ obtained using $Q$.
Under the assumption that $\beta_n \not = 0$, all subsequent results could also be 
obtained via the formulation in $H_{\mathrm{p}}(\curl_{\alpha},\Omega)^3$.
\end{remark}

\section{Integral Identities}
%\textcolor{blue}{
This section is concerned with technical lemmas that will be used to derive 
the Rellich identity and solution bounds subsequently. Roughly speaking, 
for deriving the Rellich identity, we will multiply the Maxwell 
equations~\eqref{eq:eqOmega1} by $x_3\pa H/\pa x_3$ and integrate by 
parts. Therefore, it is the aim of the technical lemmas in this 
section to analyze the term $\Re\intO x_3\pa H/\pa x_3\cdot\ov{\curl_{\alpha}(\epsr^{-1}\curl_{\alpha} H)}\d{x}$ for
a solution $H \in H^1_{\mathrm{p}, \mathrm{T}}(\Omega)^3$ to the problem~\eqref{eq:weakProblem}.
Note that the first two lemmas need the function $H$ to be in $H_{\mathrm{p}}^2(\Omega)^3$.
These lemmas for the magnetic field formulation actually correspond to the ones for the 
electric field formulation in~\cite[Section 3]{Hadda2011}.%}

We need to introduce some notation. For a vector field $F=(F_1,F_2,F_3)^{\top}$ 
we denote by $F_T=(F_1,F_2,0)^{\top}$ its transverse part. Recall that 
$\pa^{\alpha} f / \pa x_j = \pa f / \pa x_j + i \alpha_j f$ 
for a scalar function $f$ and $j=1,2,3$. Further, we introduce 
\[
  \nabla_T f := \bigg(\frac{\pa f}{\pa x_1},\frac{\pa f}{\pa x_2},0\bigg)^{\top}, \quad 
  \nalT f := \bigg(\frac{\pa^{\alpha}f}{\pa x_1},\frac{\pa^{\alpha}f}{\pa x_2},0\bigg)^{\top}, \quad
  \overrightarrow{\curl}_{\alpha,T} f
    := \bigg(\frac{\pa^{\alpha}f}{\pa x_2},-\frac{\pa^{\alpha}f}{\pa x_1},0\bigg)^{\top},
\]
and, for a vector field $F=(F_1,F_2,F_3)^{\top}$,
\[
  \div_{\alpha,T} F:=\frac{\pa^{\alpha}F_1}{\pa x_1} + \frac{\pa^{\alpha}F_2}{\pa x_2}
  \quad \text{and} \quad 
  \curl_{\alpha,T} F := \frac{\pa^{\alpha}F_2}{\pa x_1} - \frac{\pa^{\alpha}F_1}{\pa x_2}.
\] 
It is straightforwards to show that $\div_{\alpha,T} \overrightarrow{\curl}_{\alpha,T} = 0$ as well as $\curl_{\alpha,T} \nalT = 0$. 
Moreover, a tedious computation shows that 
\begin{align*}
 \curl_{\alpha} F = (\curl_{\alpha,T}  F_T)e_3 + \overrightarrow{\curl}_{\alpha,T} F_3 - \frac{\pa(F\times e_3)}{\pa x_3},
\end{align*}
%\textcolor{blue}{
and further 
\begin{align}
\label{eq:equality1}
 |\curl_{\alpha} F|^2 = |\curl_{\alpha,T}  F_T|^2 + |\overrightarrow{\curl}_{\alpha,T} F_3|^2 
+ \left|\frac{\pa F_T}{\pa x_3}\right|^2 - 2\Re\left(\ol{\nalT F_3}\cdot \frac{\pa F_T}{\pa x_3}\right).
%\left|\frac{\pa F_T}{\pa x_3}\right|^2 - \frac{\pa(F\times e_3)}{\pa x_3},
\end{align}%}
\begin{lemma}
\label{th:lemma2}
Assume that $\epsr^{-1} \in W^{1,\infty}_{\mathrm{p}}(\Omega)$ is positive 
and bounded away from zero and that $H \in H_{\mathrm{p}}^2(\Omega)^3$. Then
\begin{align}
\label{eq:eqlemma2}
  2\Re\intO & x_3\frac{\pa H}{\pa x_3}\cdot\ov{\curl_{\alpha}(\epsr^{-1}\curl_{\alpha} H)} \d{x} 
  = - \intO \frac{\pa(x_3\epsr^{-1})}{\pa x_3}|\curl_{\alpha} H|^2\d{x} 
  + h \int_{\Gamma_h}  |\curl_{\alpha} H|^2\d{s} \nonumber \\ 
  & + 2\Re \int_{\Omega} \epsr^{-1} \bigg(e_3 \times \frac{\pa H}{\pa x_3}\bigg)\cdot \ov{\curl_{\alpha} H}\d{x} 
  + 2h \Re \int_{\Gamma_h}  \frac{\pa H_T}{\pa x_3} \cdot(e_3 \times \ov{\curl_{\alpha} H})\d{s}. 
\end{align}
\end{lemma}
\begin{proof}
Denote by $\nu$ the outward unit normal to $\Omega$. Using integration by 
parts and noting that $\nu  =e_3$ on $\Gamma_h$, %\textcolor{blue}{
and that the boundary term on $\Gamma_0$ vanishes since $x_3 = 0$ on $\Gamma_0$, we find that
\begin{align*}
  2\Re\intO & x_3\frac{\pa H}{\pa x_3}\cdot\ov{\curl_{\alpha}(\epsr^{-1}\curl_{\alpha} H)}\d{x} \\
  & = 2\Re\intO \epsr^{-1} \curl_{\alpha} \bigg(x_3\frac{\pa H}{\pa x_3} \bigg)\cdot \ov{\curl_{\alpha} H}\d{x} 
  + 2\Re\int_{\pa\Omega} x_3 \frac{\pa H}{\pa x_3} \cdot(\nu \times \ov{\epsr^{-1}\curl_{\alpha} H} )\d{s} \\ 
  & =  \intO \epsr^{-1} x_3 \frac{\pa|\curl_{\alpha} H|^2}{\pa x_3}\d{x}  + 2\Re \int_{\Omega} \epsr^{-1} \bigg(e_3 \times \frac{\pa H}{\pa x_3}\bigg)\cdot \ov{\curl_{\alpha} H}\d{x} \\ 
  & \hspace*{2cm} + 2h \Re \int_{\Gamma_h}  \frac{\pa H_T}{\pa x_3}\cdot (e_3 \times \ov{\curl_{\alpha} H})\d{s}\\
  & = - \intO \frac{\pa(x_3\epsr^{-1})}{\pa x_3}|\curl_{\alpha} H|^2\d{x} + 2\Re \int_{\Omega} \epsr^{-1} \bigg(e_3 \times \frac{\pa H}{\pa x_3}\bigg) \cdot\ov{\curl_{\alpha} H}\d{x} \\ 
  & \hspace*{2cm} + h \int_{\Gamma_h}  |\curl_{\alpha} H|^2\d{s} + 2h \Re \int_{\Gamma_h}  \frac{\pa H_T}{\pa x_3} \cdot(e_3 \times \ov{\curl_{\alpha} H}) \d{s}.
\end{align*}
%The last equation is exactly equation~\eqref{eq:eqlemma2}.
\end{proof}

The next lemma continues the analysis of the term $\Re \int_{\Omega} \epsr^{-1} (e_3 \times \pa H/\pa x_3) \cdot \ov{\curl_{\alpha} H} \d{x}$
in the right hand side of~\eqref{eq:eqlemma2}.
\begin{lemma}
\label{th:lemma1}
 Assume that $\epsr^{-1} \in W^{1,\infty}_{\mathrm{p}}(\Omega)$ is positive and bounded away from zero. Then for all $H \in H_{\mathrm{p}}^2(\Omega)^3$ the following identity holds,
\begin{eqnarray}
\label{eq:eqlemma1}
2\Re \int_{\Omega} \epsr^{-1} \bigg(e_3 \times \frac{\pa H}{\pa x_3}\bigg) \cdot \ov{\curl_{\alpha} H} \d{x} 
= 2 \int_{\Omega} \epsr^{-1} \bigg|\frac{\pa H}{\pa x_3}\bigg|^2 \d{x} 
+ 2 \Re \intO\nabla \epsr^{-1}\cdot\frac{\pa H}{\pa x_3}\ov{H_3}\d{x} \nonumber \\ 
%+ \int_{\Omega} \frac{\pa \epsr^{-1}}{\pa x_3} \frac{\pa |H_3|^2}{\pa x_3}\d{x} 
- 2\Re \intO \frac{\pa(\epsr^{-1}\ov{H_3})}{\pa x_3}\div_{\alpha} H\d{x} 
- 2 \Re \int_{\Gamma_h} \bigg(\frac{\pa H_3}{\pa x_3} - \div_{\alpha} H\bigg) \ov{H_3}\d{s} \nonumber \\
- 2 \Re \int_{\Gamma_0} \epsr^{-1}\ov{H_3}\div_{\alpha,T} H_T \d{s}.
\end{eqnarray}
\end{lemma}
\begin{proof}
First, we have
\begin{eqnarray}
\label{eq:eq1}
 2\Re \int_{\Omega} \epsr^{-1} \bigg(e_3 \times \frac{\pa H}{\pa x_3}\bigg)\cdot \ov{\curl_{\alpha} H}\d{x} = 2 \int_{\Omega} \epsr^{-1} \bigg|\frac{\pa H_T}{\pa x_3}\bigg|^2\d{x} \qquad \qquad \qquad \qquad \nonumber \\ - 2 \Re \int_{\Omega} \epsr^{-1} \frac{\pa H_T}{\pa x_3} \cdot \naT \ov{H_3}\d{x} + 2 \Re \int_{\Omega} \epsr^{-1} \frac{\pa H_T}{\pa x_3} \cdot i\alpha \ov{H_3}\d{x}.
\end{eqnarray}
Second, we compute that
\begin{multline*} 
  -2 \Re \int_{\Omega} \epsr^{-1} \frac{\pa H_T}{\pa x_3} \cdot \naT \ov{H_3}\d{x} = 2 \Re \intO \div_T \bigg(\epsr^{-1} \frac{\pa H_T}{\pa x_3}\bigg) \ov{H_3}\d{x} \\ = 2 \Re \intO \epsr^{-1} \div_T \bigg(\frac{\pa H_T}{\pa x_3}\bigg) \ov{H_3}\d{x} + 2 \Re \intO \naT \epsr^{-1} \cdot \frac{\pa H_T}{\pa x_3} \ov{H_3}\d{x} \\
% = - 2 \Re \intO \pa_3(\epsr^{-1} \ov{H_3}) \div_T H_T\d{x} + 2 \Re \int_{\Gamma_h}  \ov{H_3} \div_T H_T\d{s} - 2 \Re \int_{\Gamma_0} \epsr^{-1} \ov{H_3} \div_T H_T\d{s} \\ + 2 \Re \intO \naT \epsr^{-1} \cdot \frac{\pa H_T}{\pa x_3} \ov{H_3} \d{x} \\
= -2 \Re \intO \frac{\pa \epsr^{-1}}{\pa x_3} \ov{H_3} \div_T H_T \d{x} - 2 \Re \intO \epsr^{-1} \frac{\pa\ov{H_3}}{\pa x_3} \div_T H_T\d{x} \\ + 2 \Re \intO \naT \epsr^{-1} \cdot \frac{\pa H_T}{\pa x_3} \ov{H_3}\d{x} + 2 \Re \int_{\Gamma_h}  \ov{H_3} \div_T H_T\d{s}  - 2 \Re \int_{\Gamma_0} \epsr^{-1} \ov{H_3} \div_T H_T\d{s}
\end{multline*}
Now, using the identity $\div_T H_T = - \pa H_3/\pa x_3 + \div_{\alpha} H - i\alpha \cdot H$, we obtain that
\begin{multline*} 
 -2 \Re \int_{\Omega} \epsr^{-1} \frac{\pa H_T}{\pa x_3} \cdot \naT \ov{H_3}\d{x} 
= 2 \Re\intO \frac{\pa \epsr^{-1}}{\pa x_3} \ov{H_3} (i\alpha\cdot H) \d{x} + 2 \Re \intO \frac{\pa \epsr^{-1}}{\pa x_3} \ov{H_3}\frac{\pa H_3}{\pa x_3} \d{x} \\ - 2\Re\intO \frac{\pa \epsr^{-1}}{\pa x_3}\ov{H_3} \div_{\alpha} H \d{x}  + 2\Re \intO\epsr^{-1}\frac{\pa\ov{H_3}}{\pa x_3} (i\alpha\cdot H)\d{x} + 2 \Re \intO \epsr^{-1}\bigg|\frac{\pa H_3}{\pa x_3}\bigg|^2\d{x} \\ - 2\Re\intO\epsr^{-1}\frac{\pa\ov{H_3}}{\pa x_3} \div_{\alpha} H\d{x} + 2 \Re\intO\naT\epsr^{-1} \cdot \frac{\pa H_T}{\pa x_3} \ov{H_3}\d{x} + 2 \Re \int_{\Gamma_h}  \ov{H_3} \div_T H_T \d{x} \\ - 2 \Re \int_{\Gamma_0} \epsr^{-1} \ov{H_3} \div_T H_T\d{s} 
\end{multline*}
Applying Green formula to the term $2 \Re\intO (\pa\epsr^{-1}/\pa x_3) \ov{H_3} (i \alpha\cdot H) \d{x}$, we have
% = -2\Re \intO\epsr^{-1} \ov{H_3}\frac{\pa(i\alpha\cdot H)}{\pa x_3}\d{x} + 2\Re \int_{\Gamma_h} \ov{H_3}i\alpha\cdot H\d{s} - 2\Re \int_{\Gamma_0} \epsr^{-1}\ov{H_3}i\alpha\cdot H\d{s} \qquad \,\,\\ + 2 \Re \intO\frac{\pa \epsr^{-1}}{\pa x_3}\ov{H_3}\frac{\pa H_3}{\pa x_3}\d{x} - 2\Re \intO\frac{\pa \epsr^{-1}}{\pa x_3}\ov{H_3} \div_{\alpha} H\d{x} + 2\Re\intO\epsr^{-1}\bigg|\frac{\pa H_3}{\pa x_3}\bigg|^2\d{x} \\ - 2\Re \intO\epsr^{-1}\frac{\pa\ov{H_3}}{\pa x_3}\div_{\alpha} H\d{x} + 2 \Re \intO\naT\epsr^{-1}\cdot\frac{\pa H_T}{\pa x_3}\ov{H_3}\d{x} +  2 \Re \int_{\Gamma_h}  \ov{H_3} \div_T H_T\d{s} \\ - 2 \Re \int_{\Gamma_0} \epsr^{-1} \ov{H_3} \div_T H_T \d{s}\\
% \end{eqnarray*}
%  \begin{eqnarray*}
% -2 \Re \int_{\Omega} \epsr^{-1} \frac{\pa H_T}{\pa x_3} \cdot \naT \ov{H_3}\d{x} = -2\Re \intO\epsr^{-1}\frac{\pa H_T}{\pa x_3} \cdot i\alpha \ov{H_3}\d{x} + 2\Re \int_{\Gamma_h} \frac{\pa\ov{H_3}}{\pa x_3} \frac{\pa H_3}{\pa x_3}\d{s}  \\ + 2\Re\intO\epsr^{-1}\bigg|\frac{\pa H_3}{\pa x_3}\bigg|^2\d{x} - 2\Re \intO\frac{\pa \epsr^{-1}}{\pa x_3}\ov{H_3}\div_{\alpha} H\d{x} - 2\Re \intO\epsr^{-1}\frac{\pa\ov{H_3}}{\pa x_3}\div_{\alpha} H\d{x} \\ + 2 \Re \intO\naT \epsr^{-1}\cdot\frac{\pa H_T}{\pa x_3}\ov{H_3}\d{x}  -  2 \Re \int_{\Gamma_h} \bigg(\frac{\pa H_3}{\pa x_3} - \div_{\alpha} H\bigg) \ov{H_3}\d{s}  \\+ 2 \Re \int_{\Gamma_0} \epsr^{-1}\bigg(\frac{\pa H_3}{\pa x_3} - \div_{\alpha} H\bigg) \ov{H_3}\d{s}.
% \end{eqnarray*}
\begin{multline*}
 -2 \Re \int_{\Omega} \epsr^{-1} \frac{\pa H_T}{\pa x_3} \cdot \naT \ov{H_3}\d{x} 
= -2\Re \intO\epsr^{-1}\frac{\pa H_T}{\pa x_3} \cdot i\alpha \ov{H_3}\d{x} 
%+ \int_{\Omega}\frac{\pa\epsr^{-1}}{\pa x_3} \frac{\pa |H_3|^2}{\pa x_3} \d{s} 
 + 2\Re\intO\epsr^{-1}\bigg|\frac{\pa H_3}{\pa x_3}\bigg|^2 \d{x} \\
- 2\Re \intO\frac{\pa \epsr^{-1}}{\pa x_3}\ov{H_3}\div_{\alpha} H \d{x} 
- 2\Re \intO\epsr^{-1}\frac{\pa\ov{H_3}}{\pa x_3}\div_{\alpha} H \d{x} 
+ 2 \Re \intO\nabla \epsr^{-1}\cdot\frac{\pa H}{\pa x_3}\ov{H_3} \d{x} \\
-  2 \Re \int_{\Gamma_h} \bigg(\frac{\pa H_3}{\pa x_3} - \div_{\alpha} H\bigg) \ov{H_3}\d{s} 
- 2 \Re \int_{\Gamma_0} \epsr^{-1}\ov{H_3}\div_{\alpha,T} H_T \d{s}
\end{multline*}
Now the claim follows from substituting this identity into equation~\eqref{eq:eq1}.
\end{proof}

%\textcolor{blue}{
In the following final lemma of this section we will reformulate the 
term $\Re\intO x_3\pa H/\pa x_3\cdot\ov{\curl_{\alpha}(\epsr^{-1}\curl_{\alpha} H)} \d{x}$ 
for a solution $H \in H^1_{\mathrm{p}, \mathrm{T}}(\Omega)^3$ to the problem~\eqref{eq:weakProblem}
using the last two lemmas.
\begin{lemma}
\label{th:lemma3}
Assume that $\epsr^{-1} \in W^{1,\infty}_{\mathrm{p}}(\Omega)$ is positive and bounded away from zero. 
Then any solution $H \in H^1_{\mathrm{p}, \mathrm{T}}(\Omega)^3$ to the problem~\eqref{eq:weakProblem} satisfies
\begin{align*}
  2\Re\intO  x_3\frac{\pa H}{\pa x_3}\cdot\ov{\curl_{\alpha}(\epsr^{-1}\curl_{\alpha} H)} \d{x} 
  = - \intO \frac{\pa(x_3\epsr^{-1})}{\pa x_3}|\curl_{\alpha} H|^2\d{x} 
  + h \int_{\Gamma_h}|\curl_{\alpha} H|^2\d{s} \nonumber \\ 
  + 2 \int_{\Omega} \epsr^{-1} \bigg|\frac{\pa H}{\pa x_3}\bigg|^2\d{x} + 2 \Re \intO\nabla \epsr^{-1}\cdot\frac{\pa H}{\pa x_3}\ov{H_3}\d{x} 
  - 2 \Re \int_{\Gamma_h}  \ov{H_3} \frac{\pa H_3}{\pa x_3}\d{s} \\
  + 2h \Re \int_{\Gamma_h}  \frac{\pa H_T}{\pa x_3} \cdot(e_3 \times \ov{\curl_{\alpha} H})\d{s}. 
\end{align*}
\end{lemma}%}
\begin{proof}
%\textcolor{blue}{
It is sufficient to prove that $H$ satisfies~\eqref{eq:eqlemma2} and
\begin{multline}
\label{eq:secondClaim}
2\Re \int_{\Omega} \epsr^{-1} \bigg(e_3 \times \frac{\pa H}{\pa x_3}\bigg) \cdot\ov{\curl_{\alpha} H}\d{x} 
= 2 \int_{\Omega} \epsr^{-1} \bigg|\frac{\pa H}{\pa x_3}\bigg|^2\d{x} + 2 \Re \intO\nabla \epsr^{-1}\cdot\frac{\pa H}{\pa x_3}\ov{H_3}\d{x} \\ 
 -  2 \Re \int_{\Gamma_h}  \ov{H_3} \frac{\pa H_3}{\pa x_3}\d{s}.
\end{multline}
%}

Recall that, for $h > \sup \{x_3: \, (x_1,x_2,x_3)^\top \in \supp(\epsr - 1)\}$, 
there exists a constant $0< \eta \ll 1$ such that $\epsr = 1$ in $(0,2\pi)^2\times(h-\eta,h)$. Hence, a solution 
$H \in H^1_{\mathrm{p}, \mathrm{T}}(\Omega)^3$ to the 
problem~\eqref{eq:weakProblem} belongs to $H^1_{\mathrm{p}, \mathrm{T}}(\Omega)^3 \cap H_{\mathrm{p}}^2((0,2\pi)^2\times(h-\eta,h))^3$ 
due to interior elliptic regularity theory. 
% Let $\chi_1 \in C^{\infty}(\R^2)$ such that $\chi_1 =  1$ on $(0,2\pi)^2$, $\chi_1 = 0$ on $\R^2\setminus(-1,2\pi+1)^2$ and $\chi_2 \in C^{\infty}(\R)$ such that $\chi_2 =  0$ on $(-\infty,h-\eta/2)$, $\chi_2 = 1$ on $(h-\eta/3,h)$. 
Then one can extend $H$ to a function defined in all of $\R^3$ that is $2\pi$-biperiodic
and belongs to $H_{\mathrm{p}}^1((0,2\pi)^2\times(-\infty,h))^3 \cap H_{\mathrm{p}}^2((0,2\pi)^2\times(h-\eta,\infty))^3$ (This can be seen 
using~\cite{McLea2000} combined with suitable cut-off arguments.)
% Further, the extension of $\chi_1(\chi_2H + (1-\chi_2)H)$ can be written as the product of $\chi_1$ and the extension of $\chi_2H + (1-\chi_2)H$ which implies that the extension of $H$ is in $H_{p}^1(\{x_3<h\})^3 \cap H_{p}^2(\{x_3>h-\eta\})^3$. 
By abuse of notation, we still denote the extended function by $H$. 
Let $\phi \in C^{\infty}(\R^3)$ be a smooth and non-negative function 
supported in the unit ball and $\int_{\R^3}\phi\d{x} = 1$. For 
$\delta>0$ and $x \in\R^3$ let $\phi^{\delta}(x) = \delta^{-3}\phi(x/\delta)$. 
The convolution $H^{\delta}:=\phi^{\delta}*H$ belongs to 
$H_{\mathrm{p}}^2(\Omega)^3$ and thus satisfies~\eqref{eq:eqlemma2}. 
Then, from Lemma~\ref{th:strongRelation1} and the fact that 
$H^{\delta} \rightarrow H$ in $H^1_{\mathrm{p}, \mathrm{T}}(\Omega)^3 
\cap H_{\mathrm{p}}^2((0,2\pi)^2\times(h-\eta,h))^3$ we obtain that
\begin{equation*}
  \curl_{\alpha}(\epsr^{-1}\curl_{\alpha} H^{\delta}) 
  \stackrel{\delta \to 0}{\rightarrow} 
  \curl_{\alpha}(\epsr^{-1}\curl_{\alpha} H)\quad \text{ in } L^2(\Omega)^3. 
\end{equation*}
Moreover, the convergence in $H_{\mathrm{p}}^2((0,2\pi)^2\times(h-\eta,h))^3$ 
implies that $\curl_{\alpha} H^{\delta} \rightarrow \curl_{\alpha} H$ in 
$L^2(\Gamma_h)^3$ as $\delta \to 0$. Consequently, $H$ satisfies~\eqref{eq:eqlemma2}.

It remains to show that $H$ also satisfies~\eqref{eq:secondClaim}. 
The function $H^{\delta}$ satisfies~\eqref{eq:eqlemma1} and we 
consider the limit of this identity as $\delta \rightarrow 0$. 
It is easily seen that $\div_{\alpha} H^{\delta} \rightarrow  \div_{\alpha} H = 0$ in $L^2(\Omega)$. 
% Due to Lemma~\ref{th:strongRelation2}, we know that a solution $H$ to the 
% problem~\eqref{eq:weakProblem} also solves the problem~\eqref{eq:Id0}-\eqref{eq:Id01} 
% which implies that $\epsr^{-1}e_3 \cdot H = \epsr^{-1} (e_3\cdot \curl_{\alpha} G)/k^2$ on 
%$\Gamma_0$ due to a simple calculation. 
Thus, we have
\[
  e_3\cdot H^{\delta} 
  \stackrel{\delta \to 0}{\rightarrow} 
  e_3\cdot H = 0 \quad \text{ in } H_{\mathrm{p}}^{{1/2}}(\Gamma_0), \qquad 
  \div_{\alpha,T} H_T^{\delta} 
  \stackrel{\delta \to 0}{\rightarrow} 
  \div_{\alpha,T} H_T \quad \text{ in } H_{\mathrm{p}}^{-1/2}(\Gamma_0),
\]
due to the convergence of $H^{\delta}$ to $H$ in $H_{\mathrm{p}}^1(\Omega)^3$. Further, the convergence of $H^{\delta}$ to $H$ in $H_{\mathrm{p}}^2((0,2\pi)^2\times(h-\eta,h))^3$ and the fact $\div_{\alpha} H = 0$ on $\Gamma_h$ imply that 
\[
\frac{\pa H^{\delta}}{\pa x_3} - \div_{\alpha} H^{\delta} \rightarrow \frac{\pa H_3}{\pa x_3} - \div_{\alpha} H = \frac{\pa H_3}{\pa x_3} \quad \text{ in }H_{\mathrm{p}}^{-1/2}(\Gamma_h).
\]
Plugging in these limits into~\eqref{eq:eqlemma1} shows that~\eqref{eq:secondClaim} holds.
\end{proof}

\section{Rellich Identity and Solution Estimate}
%\textcolor{blue}{
For establishing uniqueness of solution to the variational problem~\eqref{eq:weakProblem},
we derive in this section the so-called Rellich identity relating  $|\curl_{\alpha} H|^2$
and $|\pa H/\pa x_3|^2$ where $H$ is a solution to the homogeneous variational 
problem corresponding to~\eqref{eq:weakProblem}. Then, under suitable non-trapping
and smoothness conditions on the material parameter, integral inequality resulting 
from this identity allow us to obtain estimate for a solution to the homogeneous problem. 
As mentioned in the introduction, the Rellich identity and solution estimate obtained in
this section are much simpler than the ones in~\cite[Section 4]{Hadda2011}. It turns out also
that the non-trapping assumptions on the parameter material are weaker than the ones in 
the latter paper. %}

The proof of the Rellich identity is based on an integration-by-parts technique 
that goes back to Rellich~\cite{Relli1940}. Typically, this technique requires 
more regularity of a solution than just to belong to the energy space. 
In our case we will roughly speaking multiply the Maxwell equations~\eqref{eq:eqOmega1}, 
for $G=0$ in the right hand side, by $x_3\pa H/\pa x_3$ and integrate by parts.

\begin{lemma}[Rellich Identity]
\label{th:rellichIdentity}
%  Assume that $\epsr^{-1} \in W^{1,\infty}_{\mathrm{p}}(\Omega)$ is positive and bounded 
%  away from zero and that $G \in H^1_{\mathrm{p}}(\Omega)^3$. 
%  Then the following identity holds for all 
%  solutions $H \in H^1_{\mathrm{p}, \mathrm{T}}(\Omega)^3$ 
%  to problem~\eqref{eq:weakProblem},
%\textcolor{blue}{
 Assume that $\epsr^{-1} \in W^{1,\infty}_{\mathrm{p}}(\Omega)$ is positive and bounded 
 away from zero. Then the following identity holds for all 
 solutions $H \in H^1_{\mathrm{p}, \mathrm{T}}(\Omega)^3$ 
 to the homogeneous problem corresponding to~\eqref{eq:weakProblem},
\begin{align}
\label{eq:volumetric} 
  \intO \left[ 2\epsr^{-1} \bigg|\frac{\pa H}{\pa x_3}\bigg|^2 
  - x_3 \frac{\pa \epsr^{-1}}{\pa x_3}|\curl_{\alpha} H|^2
  + 2 \Re \bigg( \nabla \epsr^{-1}\cdot\frac{\pa H}{\pa x_3} \ov{H_3} \bigg) \right] \d{x} \nonumber \\
  +\, \Re \int_{\Gamma_h} e_3\times(R_\alpha\times H)\cdot \ov{H}\d{s} - 2\Re \int_{\Gamma_h} T_\alpha(H_3) \ov{H_3}\d{s}  = 0.
\end{align}%}
\end{lemma}
\begin{proof}
%\textcolor{blue}{ 
Let $H \in H^1_{\mathrm{p}, \mathrm{T}}(\Omega)^3$ be a solution 
 to the homogeneous problem corresponding to~\eqref{eq:weakProblem}. First, using integration by parts we have
\[
  \Re \int_{\Gamma_h}  \frac{\pa H_T}{\pa x_3} \cdot(e_3 \times \ov{\curl_{\alpha} H})\d{s} 
= \int_{\Gamma_h} \bigg|\frac{\pa H_T}{\pa x_3}\bigg|^2\d{s} + \Re  \int_{\Gamma_h} \frac{\pa H_T}{\pa x_3} \cdot \ov{\nalT H_3}\d{s}.
\]
Note that $H$ satisfies the assumptions of Lemma~\ref{th:lemma3}. Together with the latter equation we obtain
\begin{multline*}
 2\Re\intO x_3\frac{\pa H}{\pa x_3}\cdot\ov{\curl_{\alpha}(\epsr^{-1}\curl_{\alpha} H)}\d{x} = 
- \intO \frac{\pa(x_3\epsr^{-1})}{\pa x_3}|\curl_{\alpha} H|^2\d{x} + h \int_{\Gamma_h}  |\curl_{\alpha} H|^2 \d{s}\\ 
+ 2 \int_{\Omega} \epsr^{-1} \bigg|\frac{\pa H}{\pa x_3}\bigg|^2\d{x} + 2 \Re \intO\nabla \epsr^{-1}\cdot\frac{\pa H}{\pa x_3}\ov{H_3}\d{x}   
-  2 \Re \int_{\Gamma_h} \frac{\pa H_3}{\pa x_3} \ov{H_3} \d{s} \\
- 2h \int_{\Gamma_h} \bigg|\frac{\pa H_T}{\pa x_3}\bigg|^2\d{s} + 2h\Re  \int_{\Gamma_h} \frac{\pa H_T}{\pa x_3} \cdot \ov{\nalT H_3}\d{s}.
\end{multline*}
%Let us set $\tilde{G} = \curl_{\alpha} ((1-\epsr^{-1}) G) \in L^2(\Omega)^3$ and
%exploit that $H$ solves~\eqref{eq:eqOmega1}, 
We exploit that $H$ solves~\eqref{eq:eqOmega1} for $G=0$,
\begin{multline*}
2\Re\intO x_3\frac{\pa H}{\pa x_3}\cdot\ov{\curl_{\alpha}(\epsr^{-1}\curl_{\alpha} H)}\d{x} = k^2 2 \Re\intO x_3\frac{\pa \ov{H}}{\pa x_3}\cdot H \d{x} 
%+ 2\Re \intO x_3\frac{\pa \ov{H}}{\pa x_3}\cdot \tilde{G}\d{x} \\ 
  = k^2 \intO x_3\frac{\pa |H|^2}{\pa x_3} \d{x}\\ %+ 2\Re \intO x_3\frac{\pa \ov{H}}{\pa x_3}\cdot \tilde{G}\d{x} \\ 
  = -k^2 \intO|H|^2\d{x} + k^2 h \int_{\Gamma_h} |H|^2\d{s}. %+ 2\Re \intO x_3\frac{\pa \ov{H}}{\pa x_3}\cdot \tilde{G}\d{x}.
\end{multline*}
From the  last two equations we conclude that
\begin{multline*}
  - \intO \bigg(\frac{\pa(x_3\epsr^{-1})}{\pa x_3}|\curl_{\alpha} H|^2 - k^2|H|^2\bigg)\d{x} 
+ 2 \int_{\Omega} \epsr^{-1} \bigg|\frac{\pa H}{\pa x_3}\bigg|^2\d{x} + 2 \Re \intO\nabla \epsr^{-1}\cdot\frac{\pa H}{\pa x_3}\ov{H_3}\d{x}  \\
-  2 \Re \int_{\Gamma_h}  \ov{H_3} \frac{\pa H_3}{\pa x_3}\d{s}  - 2h \int_{\Gamma_h} \bigg|\frac{\pa H_T}{\pa x_3}\bigg|^2\d{s} 
+ 2h \Re \int_{\Gamma_h} \frac{\pa H_T}{\pa x_3} \cdot \ov{\nalT H_3}\d{s} \\
+ h \int_{\Gamma_h} ( |\curl_{\alpha} H|^2 - k^2|H|^2)\d{s} = 0. %2\Re \intO x_3\frac{\pa \ov{H}}{\pa x_3}\cdot \tilde{G}\d{x}.
\end{multline*}%}
%\textcolor{blue}{
Due to the variational formulation~\eqref{eq:weakProblem} for $G = 0$,
\begin{equation}
  \label{eq:usefullVarForm}
 \intO (\epsr^{-1} |\curl_{\alpha} H|^2 - k^2 |H|^2)\d{x} + \Re \int_{\Gamma_h} e_3\times(R_\alpha\times H)\cdot \ov{H}\d{s} 
= 0 %\Re \intO  (1-\epsr^{-1})G \cdot \ov{\curl_\alpha H} \d{x}
\end{equation}
since $\div_\alpha H = 0$ in $\Omega$ and $R_\alpha \cdot H = 0$ in 
$H^{-1/2}_{\mathrm{p}}(\Gamma_h)$ due to Lemma~\ref{th:strongRelation1}.
Adding the last two equations yields that the term $\intO k^2 |H|^2\d{x}$ cancels,
and further exploiting $\partial H_3 / \partial x_3 = T_\alpha H_3$ on $\Gamma_h$ 
to yields that  
\begin{multline*}
- \intO x_3 \frac{\pa \epsr^{-1}}{\pa x_3}|\curl_{\alpha} H|^2\d{x}  
  + 2 \int_{\Omega} \epsr^{-1} \bigg|\frac{\pa H}{\pa x_3}\bigg|^2\d{x} 
  + 2 \Re \intO\nabla \epsr^{-1}\cdot\frac{\pa H}{\pa x_3}\ov{H_3}\d{x} \\ 
%+ \intO \frac{\pa \epsr^{-1}}{\pa x_3} \frac{\pa |H_3|^2}{\pa x_3}\d{x} 
  - 2 \Re \int_{\Gamma_h}  T_\alpha(H_3) \ov{H_3} \d{s} 
+ \Re  \int_{\Gamma_h} e_3\times(R_\alpha\times H)\cdot \ov{H}\d{s}  
  + 2h \Re \int_{\Gamma_h} \frac{\pa H_T}{\pa x_3} \cdot \ov{\nalT H_3}\d{s}\\ 
  + h \int_{\Gamma_h} \bigg(|\curl_{\alpha} H|^2 - k^2|H|^2- 2 \bigg|\frac{\pa H_T}{\pa x_3}\bigg|^2 \bigg)\d{s} 
= 0.
%2\Re \intO x_3\frac{\pa \ov{H}}{\pa x_3}\cdot \tilde{G}\d{x} 
%  + \Re \intO (1-\epsr^{-1}) G \cdot \ov{\curl_\alpha H} \d{x}.
\end{multline*}
Recall equality~\eqref{eq:equality1},
\begin{equation*}
 |\curl_{\alpha} H|^2 = |\curl_{\alpha,T} H|^2  + |\overrightarrow{\curl}_{\alpha,T} H_3|^2 
+ \bigg|\frac{\pa H_T}{\pa x_3}\bigg|^2 - 2 \Re\bigg(\frac{\pa H_T}{\pa x_3} \cdot \ov{\nalT H_3}\bigg).
\end{equation*}
Combining the last two equations yields 
\begin{align*} 
  L(H) =   h\int_{\Gamma_h} \bigg(\bigg|\frac{\pa H_T}{\pa x_3}\bigg|^2 + k^2|H|^2
    - |\curl_{\alpha,T} H|^2 - |\overrightarrow{\curl}_{\alpha,T} H_3|^2 \bigg)\d{s}
%\Re \intO \left[ 2x_3\frac{\pa \ov{H}}{\pa x_3}\cdot \curl_{\alpha} ((1-\epsr^{-1}) G)
%     + (1-\epsr^{-1}) G \cdot \ov{\curl_\alpha H}\right]\d{x} .
\end{align*}
where $L(H)$ is the left hand side of~\eqref{eq:volumetric}. It remains now to prove that the right hand
side of the latter equation vanishes.
% \begin{align}
%  \int_{\Gamma_h} \bigg(\bigg|\frac{\pa H_T}{\pa x_3}\bigg|^2 + k^2|H|^2
%     - |\curl_{\alpha,T} H|^2 - |\overrightarrow{\curl}_{\alpha,T} H_3|^2 \bigg)\d{s} = 0.
% \end{align}
First, we recall from Lemma~\ref{th:strongRelation1} that
$\partial H/\partial x_3 = T_\alpha H$ in $H^{-1/2}_\text{p}(\Gamma_h)$ which yields that
\[
 \int_{\Gamma_h}\bigg|\frac{\pa H_T}{\pa x_3}\bigg|^2 = \sum_{n\in\Lambda} |\beta_n \hat{H}_{T,n}|^2, \quad
\int_{\Gamma_h}\bigg|\frac{\pa H_3}{\pa x_3}\bigg|^2 = \sum_{n\in\Lambda} |\beta_n \hat{H}_{3,n}|^2.
\]
Using the latter formulas and replacing $k^2$ by $|n+\alpha|^2+\beta_n^2$ in the first boundary term 
in~\eqref{eq:volumetric} yields 
\begin{align}
\label{eq:boundary1}
  &\int_{\Gamma_h} \bigg(\bigg|\frac{\pa H_T}{\pa x_3}\bigg|^2 + k^2|H|^2
    - |\curl_{\alpha,T} H|^2 - |\overrightarrow{\curl}_{\alpha,T} H_3|^2 \bigg)\d{s} \nonumber\\
  &= \sum_{n\in\Lambda} \Big[ |\beta_n \hat{H}_{T,n}|^2 + (|n+\alpha|^2+\beta_n^2)(|\hat{H}_{T,n}|^2 + |\hat{H}_{3,n}|^2)
 - |(n+\alpha)\times\hat{H}_{T,n}|^2 - |n+\alpha|^2|\hat{H}_{3,n}| \Big]\nonumber\\
  &= \sum_{n\in\Lambda} \Big[ (\beta_n^2+|\beta_n|^2)|\hat{H}_{T,n}|^2 + |n+\alpha|^2|\hat{H}_{T,n}|^2
  - |(n+\alpha)\times\hat{H}_{T,n}|^2 + \beta_n^2|\hat{H}_{3,n}| \Big].
%    - |(n_1+\alpha_1)\hat{H}_{2,n} -  (n_2+\alpha_2)\hat{H}_{1,n}|^2 \\ 
%    - |(n_2+\alpha_2)\hat{H}_{3,n}|^2 - |(n_1+\alpha_1)\hat{H}_{3,n}|^2\bigg].
\end{align}
On the other hand, due to the divergence-free condition, we have
\begin{align*}
 \sum_{n\in\Lambda} \Big[ |n+\alpha|^2|\hat{H}_{T,n}|^2
  - |(n+\alpha)\times\hat{H}_{T,n}|^2 \Big] 
= \sum_{n\in\Lambda} |(n_1+\alpha_1)\hat{H}_{1,n} + (n_2+\alpha_2)\hat{H}_{2,n}|^2\\ 
= \|\div_{\alpha,T} H_T\|^2_{L^2(\Gamma_h)} 
= \|\pa H_3/\pa x_3\|^2_{L^2(\Gamma_h)} =  \sum_{n\in\Lambda} |\beta_n \hat{H}_{3,n}|^2.
\end{align*}
Now substituting the latter equation into~\eqref{eq:boundary1} leads to 
% Note that
% \begin{eqnarray*}
%  \sum_{n\in\Lambda} |(n_1+\alpha_1)\hat{H}_{1,n} +  (n_2+\alpha_2)\hat{H}_{2,n}|^2 = \|\div_{\alpha,T} H_T\|^2_{L^2(\Gamma_h)} 
% = \bigg\|\frac{\pa H_3}{\pa x_3}\bigg\|^2_{L^2(\Gamma_h)} =  \sum_{n\in\Lambda} |\beta_n \hat{H}_{3,n}|^2.
% \end{eqnarray*}
%In consequence,
\begin{align}
\label{eq:boundary2}
  \int_{\Gamma_h} \bigg(\bigg|\frac{\pa H_T}{\pa x_3}\bigg|^2 + k^2|H|^2
    - |\curl_{\alpha,T} H|^2 - |\overrightarrow{\curl}_{\alpha,T} H_3|^2 \bigg)\d{s} 
%&= \sum_{n\in\Lambda} (|\beta_n|^2 + \beta_n^2 )|\hat{H}_n|^2 \nonumber\\
  = 2\sum_{\beta_n \geq 0} \beta_n^2 |\hat{H}_n|^2,
\end{align}
where we exploited that $\beta_n$ is either a non-negative or a purely imaginary number. 
% since $\beta_n^2 = k^2 - |n+\alpha|^2 \leq k \sqrt{k^2 - |n+\alpha|^2} = k\beta_n$ 
% if $\beta_n \geq 0$ (that is, $k^2 \geq |n+\alpha|^2$).
The proof is hence finished if we show that 
$\sum_{\beta_n \geq 0} \beta_n^2 |\hat{H}_n|^2 = 0$ (since then $L(H)=0$, 
which is the claim of the theorem). First, we compute that 
\begin{align*}
\langle e_3\times(R_\alpha\times H), H \rangle_{\Gamma_h}  &=  \sum_{n \in \Lambda} i(n+\alpha)\cdot \ov{\hat{H}_{T,n}}\hat{H}_{3,n} - \sum_{n \in \Lambda} i \beta_n|\hat{H}_{T,n}|^2  \\ 
&= -\sum_{n \in \Lambda} i \ov{\beta_n}|\hat{H}_{3,n}|^2 - \sum_{n \in \Lambda} i \beta_n|\hat{H}_{T,n}|^2.
\end{align*}
Since $\Re (\beta_n) \geq 0$ this implies that 
\begin{align}
\Im \langle e_3\times(R_\alpha\times H), H \rangle_{\Gamma_h} &= -\sum_{n \in\Lambda} \Re(\beta_n)|\hat{H}_n|^2 \leq 0, \text{ and} \label{eq:eq4} \\
\Re \langle e_3\times(R_\alpha\times H), H \rangle_{\Gamma_h} &= 
\sum_{n \in \Lambda} \Im(\ov{\beta_n})|\hat{H}_{3,n}|^2 + \sum_{n \in \Lambda}\Im(\beta_n) |\hat{H}_{T,n}|^2.  \label{eq:eq5}
\end{align}
(The second equation will be exploited later on.)
Taking the imaginary part of the variational formulation~\eqref{eq:weakProblem} 
with $G=0$ and $F=H$, and exploiting Lemma~\ref{th:strongRelation1}, we obtain that
\begin{equation*}
  0  = \Im\langle e_3\times(R_\alpha\times H),H \rangle_{\Gamma_h} 
  \stackrel{\eqref{eq:eq4}}{=} - \sum_{n \in\Lambda} \Re(\beta_n)|\hat{H}_n|^2.
\end{equation*}
This implies that $|\hat{H}_n|^2 = 0$ for all $n$ such that $\Re(\beta_n)>0$. 
Since $\beta_n$ is either purely imaginary or non-negative,  
we conclude that $\sum_{\beta_n \geq 0} \beta_n^2 |\hat{H}_n|^2 = 0$.%}
\end{proof}

The next Poincar\'{e}-like result is classical (see, e.g.,~\cite{Chand2005} for a proof).
\begin{lemma}
\label{th:poincare}
 For $u  \in \{v\in H^1_{\mathrm{p}}(\Omega): v|_{\Gamma_0} = 0\}$ there holds
 $2\|u\|^2_{L^2(\Omega)} \leq h^2 \| \pa u / \pa x_3 \|^2_{L^2(\Omega)}$. 
\end{lemma}

The following assumptions on $\epsr^{-1}$ will guarantee a stability estimate 
and a uniqueness statement for a solution to the variational problem~\eqref{eq:weakProblem}:
\begin{equation}
\label{eq:assumptions}
\begin{split}
& (a)\quad  \epsr^{-1} \in W^{1,\infty}_{\mathrm{p}}(\Omega) \text{ satisfies } \frac{\pa\epsr^{-1}}{\pa x_3} \leq 0 \text{ in } \Omega,\\
& (b) \quad \text{It holds that } \frac{\pa\epsr^{-1}}{\pa x_3} < 0 \text{ in a non-empty open ball } B \subset \Omega,\\
& (c) \quad\text{There exists } \delta>1/2 \text{ such that } \frac{\delta}{2} \|\naT \epsr^{-1}\|_{L^{\infty}(\Omega)^2}^2 + \frac{\sqrt{2}}{h}\bigg\| \frac{\pa \epsr^{-1}}{\pa x_3} \bigg\|_{L^\infty(\Omega)} < \frac{2}{h^2}.
% & (c) \quad\text{There exists }\sigma > \sqrt{\frac{2}{c}} : \frac{c}{h^2} - \bigg(\sigma\frac{\pa\epsr^{-1}}{\pa x_3}\bigg)^2 \geq 0,\\
% & (c) \quad\text{It holds that } \left( \frac{\pa\epsr^{-1}}{\pa x_3} \right)^2 < \frac{1}{\sqrt{2} \, h} \text{ and } \|\naT \epsr^{-1}\|_{L^{\infty}(\Omega)^2} \leq \frac{1}{\sqrt{2}\, h}. 
\end{split}
\end{equation}

\begin{remark}
Note that~\eqref{eq:assumptions}(a) implies that $\epsr^{-1} \geq 1$, 
since, by construction, $\epsr^{-1}=1$ in $\{ h-\eta < x_3 < h\}$ for some small $\eta>0$. 
For the case of periodic non-absorbing structures, the main difference between these 
non-trapping conditions and the ones for the scalar case in~\cite{Bonne1994} is 
the additional condition~\eqref{eq:assumptions}(c). This condition arises from estimating the term 
$2 \Re \intO(\nabla \epsr^{-1}\cdot\pa H/\pa x_3 \ov{H_3})\d{x}$ in the Rellich 
identity~\eqref{eq:volumetric} using the Poincar\'{e}-like result above. This is 
natural since the Rellich identity resulting from a similar technique for
the scalar case~\cite{Bonne1994} does not have a corresponding term. 

%\textcolor{blue}{
Let us construct a function $\epsr^{-1}$ that satisfies the above assumptions~\eqref{eq:assumptions}. 
Choose constants $0<h_1<h_2<h$, $\lambda>0$, and a $C^1$-smooth cut-off function $\chi\in C^{1}((0,2\pi)^2)$ 
with compact support in $(0,2\pi)^2$ such that $0\leq \chi\leq 1$ and $\chi = 1$ in 
$(\pi/2,3\pi/2)^2$. For $x = (x_1,x_2,x_3)^\top\in \Omega$, we define
\[  \epsr^{-1}(x_1,x_2,x_3)= \begin{cases} 
              \lambda\chi(x_1,x_2) + 1 , & 0<x_3<h_1, \\ 
              \lambda\left(\frac{x_3-h_2}{h_1-h_2}\right)\chi(x_1,x_2) + 1, & h_1<x_3<h_2, \\
               1, & h_2<x_3<h.
            \end{cases}  \] 
%with $2\pi$-periodic extensions in $x_1$ and $x_2$. 
Then $\epsr^{-1}$ is a decreasing function that satisfies~\eqref{eq:assumptions}(a), 
and condition~\eqref{eq:assumptions}(c) is satisfied 
when $\lambda>0$ is small enough. Moreover, 
$\epsr^{-1}$ also satisfies condition~\eqref{eq:assumptions}(b) in $(\pi/2,3\pi/2)^2 \times (h_1, h_2)$.
However, $\epsr^{-1}$ does not satisfy the corresponding conditions (7.2)(b,c) in~\cite{Hadda2011}, 
which require, roughly speaking, strict positivity of $\pa\epsr/\pa x_3$ in $(0,2\pi)^2 \times (h_1, h_2)$ 
(an arbitrary ball $B \subset \Omega$ as in~\eqref{eq:assumptions}(b) is not sufficient for the proof in \cite{Hadda2011}).%}
\end{remark}

\begin{lemma}
\label{th:estimatelemma}
Assume that $\epsr^{-1}$ satisfies the three assumptions in~\eqref{eq:assumptions}. Then there 
exists $C>0$ (independent of $k>0$) such that 
\begin{align*}
C \intO \bigg|\frac{\pa H}{\pa x_3}\bigg|^2\d{x} \leq \intO x_3\frac{\pa\epsr^{-1}}{\pa x_3}|\curl_{\alpha} H|^2\d{x} 
%(2kh+2h +1) \, \| H \|_{H^1_{\mathrm{p}}(\Omega)^3} \| (1-\epsr^{-1}) G \|_{H^1_{\mathrm{p}}(\Omega)^3}
\end{align*}
for all solutions $H \in H_{\mathrm{p},\mathrm{T}}^1(\Omega)^3$ to the homogeneous problem corresponding to~\eqref{eq:weakProblem}.
\end{lemma}
\begin{proof}
%\textcolor{blue}{
We first estimate the two boundary terms in~\eqref{eq:volumetric}. We find that 
\begin{equation*}
-2\Re \int_{\Gamma_h}T_\alpha(H_3)\overline{H_3}\d{s} = 2 \sum_{n\in\Lambda} \Im(\beta_n)|\hat{H}_{3,n}|^2 \geq 0.
\end{equation*}
Together with~\eqref{eq:eq5} we obtain
\[
 \Re \langle e_3\times(R_\alpha\times H), H \rangle_{\Gamma_h} -2\Re \int_{\Gamma_h}T_\alpha(H_3)\overline{H_3}\d{s} 
=  \sum_{n\in\Lambda} \Im(\beta_n)|\hat{H}_n|^2 \geq 0. 
\]
Therefore, from the Rellich identity~\eqref{eq:volumetric} we deduce $V(H) \leq 0$ 
where $V(H)$ is the volumetric terms in~\eqref{eq:volumetric}. 
We need now to bound $V(H)$ from below,%}
\begin{multline*}
  V(H) = \intO \left[ 2\epsr^{-1} \bigg|\frac{\pa H}{\pa x_3}\bigg|^2 
   - x_3 \frac{\pa \epsr^{-1}}{\pa x_3}|\curl_{\alpha} H|^2
   +2 \Re \bigg( \naT \epsr^{-1}\cdot\frac{\pa H_T}{\pa x_3}\ov{H_3}
   + \frac{\pa \epsr^{-1}}{\pa x_3}\frac{\pa H_3}{\pa x_3} \ov{H_3}\bigg) \right] \d{x}\\
 \geq  \intO \left[ 2 \bigg|\frac{\pa H}{\pa x_3}\bigg|^2
   - x_3\frac{\pa \epsr^{-1}}{\pa x_3} |\curl_{\alpha} H|^2\d{x}\right]\d{x} 
- \gamma^{-1} \left\| \frac{\pa H_3}{\pa x_3} \right\|_{L^2(\Omega)}^2
   - \gamma \left\| \frac{\pa \epsr^{-1}}{\pa x_3} \right\|^2_{L^\infty(\Omega)}\left\|H_3 \right\|_{L^2(\Omega)}^2\\
 - \delta \|\naT \epsr^{-1}\|_{L^{\infty}(\Omega)^2}^2 
   \| H_3\|_{L^2(\Omega)}^2 
 - \delta^{-1} \bigg\|\frac{\pa H_T}{\pa x_3}\bigg\|_{L^2(\Omega)^2}^2 
\end{multline*}
for arbitrary $\delta, \,\gamma>0$. Poincar\'e's inequality from 
Lemma~\ref{th:poincare} and the binomial formula imply that
\begin{multline*}
  V(H) \geq  \intO \left[ \big(2- \frac{\delta h^2}{2} \|\naT \epsr^{-1}\|_{L^{\infty}(\Omega)^2}^2 \big) 
   \bigg|\frac{\pa H_3}{\pa x_3}\bigg|^2
   + \frac{2\delta-1}{\delta} \bigg|\frac{\pa H_T}{\pa x_3}\bigg|^2 
   - x_3\frac{\pa \epsr^{-1}}{\pa x_3} |\curl_{\alpha} H|^2\d{x}\right]\d{x} \\
% + \intO \bigg| \gamma^{-1}\frac{\pa \ol{H_3}}{\pa x_3}  + \gamma \frac{\pa \epsr^{-1}}{\pa x_3} H_3 \bigg|^2 \d{x}
- \gamma^{-1} \left\| \frac{\pa H_3}{\pa x_3} \right\|_{L^2(\Omega)}^2
   - \gamma \left\| \frac{\pa \epsr^{-1}}{\pa x_3} \right\|^2_{L^\infty(\Omega)}\left\|H_3 \right\|_{L^2(\Omega)}^2
\end{multline*}
Again, we exploit Poincar\'e's inequality, to find that 
\[
   \gamma^{-1} \left\| \frac{\pa H_3}{\pa x_3} \right\|_{L^2(\Omega)}^2
   + \gamma \left\| \frac{\pa \epsr^{-1}}{\pa x_3} \right\|^2_{L^\infty(\Omega)}\left\|H_3 \right\|_{L^2(\Omega)}^2
  \leq \left( \gamma^{-1} + \gamma \frac{h^2}{2} \left\| \frac{\pa \epsr^{-1}}{\pa x_3} \right\|^2_{L^\infty(\Omega)} \right)
    \left\| \frac{\pa H_3}{\pa x_3} \right\|_{L^2(\Omega)}^2.
\]
The minimum of $\gamma \mapsto \gamma^{-1} + C \gamma$ is $2\sqrt{C}$. In consequence,  
\begin{multline*}
   V(H) \geq \left[ 2 -  \frac{\delta h^2}{2} \|\naT \epsr^{-1}\|_{L^{\infty}(\Omega)^2}^2 
   - \sqrt{2}h \left\| \frac{\pa \epsr^{-1}}{\pa x_3} \right\|_{L^\infty(\Omega)}\right] \intO \bigg|\frac{\pa H_3}{\pa x_3}\bigg|^2 \d{x} \\
   + \frac{2\delta-1}{\delta} \intO \bigg|\frac{\pa H_T}{\pa x_3}\bigg|^2 \d{x}
   - \intO x_3 \frac{\pa \epsr^{-1}}{\pa x_3} |\curl_{\alpha} H|^2 \d{x}.
\end{multline*}
Finally, assumption~\eqref{eq:assumptions}(c) implies that there exists 
$\delta>1/2$ such that the first bracket on the right-hand side is positive.
\end{proof}

\section{Uniqueness of Solution for All Wave Numbers}

In this section, we prove our main uniqueness result for the 
electromagnetic scattering problem~\eqref{eq:weakProblem}, 
under the assumption that $\epsr$ satisfies~\eqref{eq:assumptions}.
As mentioned above, corresponding uniqueness results that hold 
for all wave numbers currently exist, to the best of our 
know\-ledge, only for absorbing materials, see~\cite{Schmi2003}, 
or simpler two-dimensional structures, see~\cite{Bonne1994}. 

\begin{theorem}
  \label{th:uniqueness}
  Assume that $\epsr^{-1}$ satisfies the assumptions~\eqref{eq:assumptions}. 
  Then problem~\eqref{eq:weakProblem} is uniquely solvable for all right-hand 
  sides $G \in H^1_{\mathrm{p}}(\Omega)$ and for all wave numbers $k>0$.
\end{theorem}
\begin{proof}
Consider a solution $H \in H^1_{\mathrm{p},\mathrm{T}}(\Omega)^3$ to the 
homogeneous problem corresponding  to~\eqref{eq:weakProblem}. 
Due to Lemma~\ref{th:estimatelemma} and the assumptions on $\epsr^{-1}$ we obtain that
 $\pa H / \pa x_3 = 0$ in $\Omega$ and $\curl_{\alpha} H = 0$ in the ball $B$
(see assumption~\eqref{eq:assumptions}(b)). Equation~\eqref{eq:eqOmega1} implies that 
$H$ vanishes in $B$, too. 

Since $H$ is independent of $x_3$, it is sufficient to show that $H$ vanishes on 
$\Gamma_{h-\eta} = \{(x_1,x_2,x_3) \in \Omega: x_3 = h-\eta \}$ for some (small) $\eta>0$
to conclude that $H$ vanishes entirely in $\Omega$.
If $\eta$ is small enough, then all three components $H_j$, $j=1,2,3$, satisfy 
\[
  \Delta_{\alpha} H_j + k^2 H_j = 0, \qquad \Delta_{\alpha} H_j := \Delta H_j + 2i \alpha \cdot \nabla H_j - |\alpha|^2 H_j, 
\]
in some neighborhood of $\Gamma_{h-\eta}$. 
Let us denote by $\Delta_2 = \partial^2  / \partial x_1^2 + \partial^2  / \partial x_2^2$ the 
two-dimensional Laplacian. Since $\partial^2 H_j  / \partial x_3^2$ vanishes, 
$\left. H_j \right|_{\Gamma_{h-\eta}} \in H^1_{\mathrm{p}}(\Gamma_{h-\eta})$ 
is a weak solution to the two-dimensional equation  
\[
  \Delta_2 H_j + 2i \alpha \cdot \nabla_T H_j + (k^2 - |\alpha|^2) H_j = 0 \quad \text{on } \Gamma_{h - \eta}, \quad j=1,2,3.
\]
Standard elliptic regularity results imply that $\left. H_j \right|_{\Gamma_{h-\eta}}$ belongs to 
$H^2_{\mathrm{p}}(\Gamma_{h-\eta})$. Moreover, since $H$ vanishes in the open 
ball $B$ and since $H$ is independent of $x_3$, $H_j$ 
vanish in a non-empty relatively open subset of $\Gamma_{h - \eta}$. 

In this situation, the unique continuation principle stated in Theorem~\ref{th:Kirsch} 
(see, e.g.,~\cite{Monk2003a}) implies that $H_{j}$ vanishes on $\Gamma_{h-\eta}$ for $j=1,2,3$,  
and hence $H$ vanishes in $\Omega$.
\end{proof}

\begin{theorem}
\label{th:Kirsch}
 Let $\O$ be an open and simply connected set in $\R^2$, 
 and let $u_1,...,u_m \in H^2(\O)$ be real-valued such that
\begin{equation}
 |\Delta u_j| \leq C \sum_{l = 1}^m (|u_l| + |\nabla u_l|) \text{ in } \O \text{ for } j=1,...,m.
\end{equation}
If $u_j$ vanishes in some open and non-empty subset of $\O$ for all $j = 1,...,m$, then $u_j$ vanish identically in $\O$ for all $j = 1,...,m$.
\end{theorem}

\bibliographystyle{siam}
\bibliography{ip-biblio}

\end{document}